\documentclass[a4paper,12pt]{article}

\usepackage{mymacros}

\newcommand{\psihat}{{\hat{\psi}}}
\newcommand{\Gammaa}{\Gamma_{n+2}}
\newcommand{\Gammab}{\Gamma}
\newcommand{\Fa}{\scF_{n+2}}

\newcommand{\Fqb}{\scF_q}
\newcommand{\Sa}{\scS_{n+2}}
\newcommand{\Sb}{\scS}

\newcommand{\Sqb}{\scS_q}

\newcommand{\fo}{\Lambda_\bN / q^2 \Lambda_{\bN}}

\newcommand{\ofo}{\otimes_{\Lambda_\bN} \fo}

\newcommand{\PSL}{\mathop{PSL}\nolimits}

\newcommand{\Sing}{\operatorname{Sing}}

\renewcommand{\Fuk}{\scF}
\renewcommand{\dirFuk}{\Fuk^\to}

\newtheorem{thm}{Theorem}[section]
\newtheorem{prop}[thm]{Proposition}

\newtheorem{lem}[thm]{Lemma}
\newtheorem{ass}[thm]{Assumption}

\theoremstyle{definition}

\newcommand{\ii}{\sqrt{-1}}
\renewcommand{\red}{\mathrm{red}}

\newcommand{\dirC}{C^{\to}}

\newcommand{\Ls}{L^\sharp}
\newcommand{\Ld}{L^\Diamond}

\title{Homological mirror symmetry\\
for the quintic 3-fold}
\author{Yuichi Nohara and Kazushi Ueda}
\date{}
\begin{document}

\maketitle

\begin{abstract}
We prove homological mirror symmetry
for the quintic Calabi-Yau 3-fold.
The proof follows that
for the quartic surface
by Seidel \cite{Seidel_K3} closely, and
uses a result of
Sheridan \cite{Sheridan_pants}.
In contrast to Sheridan's approach
\cite{Sheridan_CY},
our proof gives the compatibility
of homological mirror symmetry
for the projective space
and its Calabi-Yau hypersurface.
\end{abstract}

\section{Introduction}
 \label{sc:introduction}

Ever since the proposal by Kontsevich
\cite{Kontsevich_HAMS},
homological mirror symmetry has been proved
for elliptic curves \cite{Polishchuk-Zaslow,
Polishchuk_MFPEC,
Seidel_ASNT},
Abelian surfaces
\cite{Fukaya_MSAVMTF,
Kontsevich-Soibelman_HMSTF,
Abouzaid-Smith}
and quartic surfaces
\cite{Seidel_K3}.
It has also been extended
to other contexts
such as Fano varieties
\cite{Kontsevich_ENS98},
varieties of general type
\cite{Katzarkov_BGHMS},
and singularities
\cite{Takahashi_WPL},
and various evidences
have been accumulated in each cases.

The most part of the proof of homological mirror symmetry for the quartic surface
by Seidel \cite{Seidel_K3} works
in any dimensions.
Combined with the results of Sheridan \cite{Sheridan_pants},
an expert reader will observe that
one can prove homological mirror symmetry
for the quintic 3-fold
if one can show that
\begin{itemize}
 \item
the large complex structure limit monodromy
of the pencil of quintic Calabi-Yau 3-folds is {\em negative}
in the sense of Seidel \cite[Definition 7.1]{Seidel_K3}, and
 \item
the vanishing cycles of the pencil of quintic Calabi-Yau 3-folds
are isomorphic in the Fukaya category
to Lagrangian spheres
constructed by Sheridan \cite{Sheridan_pants}.
\end{itemize}
We prove these statements, and
obtain the following:

\begin{theorem} \label{th:hms}
Let $X_0$ be a smooth quintic Calabi-Yau 3-fold in $\bP_\bC^4$
and $Z_q^*$ be the mirror family.
Then there is a continuous automorphism
$\psi \in \End(\Lambda_\bN)^\times$
and an equivalence
\begin{equation} \label{eq:hms}
 D^\pi \Fuk(X_0) \cong \psihat^* D^b \coh Z_q^*
\end{equation}
of triangulated categories over $\Lambda_\bQ$.
\end{theorem}

Here $\Lambda_\bN = \bC[\![q]\!]$ is the ring of
formal power series in one variable and
$
 \Lambda_\bQ
$
is its algebraic closure.
The automorphism $\psihat$ of $\Lambda_\bQ$
is any lift of the automorphism
$\psi$ of $\Lambda_\bN$,
and the category $\psihat^* D^b \coh Z_q^*$ is obtained
from $D^b \coh Z_q^*$
by changing the $\Lambda_\bQ$-module structure
by $\psihat$.
The category $D^\pi \Fuk(X_0)$ is the split-closed
derived Fukaya category of $X_0$
consisting of rational Lagrangian branes.
The symplectic structure of $X_0$
and hence the parameter $q$ come from
$5$ times the Fubini-Study metric
of the ambient projective space $\bP_\bC^4$.
The mirror family $Z_q^* = [Y_q^* / \Gammab]$ is
the quotient of
the hypersurface
$$
 Y_q^* =
  \{[y_1 : \cdots : y_5] \in \bP^{4}_{\Lambda_\bQ} \mid
   y_1 \cdots y_{5} + q( y_1^{5} + \cdots + y_{5}^{5})
     = 0 \}
$$
by the group
\begin{equation} \label{eq:Gamma}
 \Gamma = \{ [\diag(a_1, \ldots, a_{5})]
  \in \PSL_{5}(\bC) \mid
  a_1^{5} = \cdots = a_{5}^{5}
   = a_1 \cdots a_{5}
   = 1 \}.
\end{equation}
Let $Z_q = [ Y_q / \Gamma]$ be the quotient
of the hypersurface $Y_q$ of $\bP_{\Lambda_\bN}^4$
defined by the same equation as $Y_q^*$ above.
The equivalence \eqref{eq:hms} is obtained by
combining the equivalences
$$
 D^\pi \Fuk(X_0)
  \cong \psihat^* D^\pi \scS_q^*
  \cong \psihat^* D^b \coh Z_q^*
$$
for an $A_\infty$-algebra
$\scS_q^* = \scS_q \otimes_{\Lambda_\bN} \Lambda_\bQ$
as follows:
\begin{enumerate}
 \item
The derived category $D^b \coh Z_q^*$
of coherent sheaves on $Z_q^*$ has
a split-generator,
which extends to an object of $D^b \coh Z_q$.
The quasi-isomorphism class
of the endomorphism dg algebra $\Sqb$
of this object is characterized
by its cohomology algebra
together with a couple of additional properties
up to pull-back by $\End(\Lambda_\bN)^\times$.
 \item
The Fukaya category $\Fuk(X_0)$ contains
$625$ distinguished Lagrangian spheres.
They are vanishing cycles
for a pencil of quintic Calabi-Yau 3-folds,
and a suitable combination of symplectic Dehn twists along them
is isotopic to the {\em large complex structure limit monodromy}.
 \item
The large complex structure limit monodromy
has a crucial property of {\em negativity},
which enables one to show that
the vanishing cycles split-generate
the derived Fukaya category $D^\pi \Fuk(X_0)$.
 \item
The total morphism $A_\infty$-algebra
$\Fqb$
of the vanishing cycles
has the same cohomology algebra as $\scS_q$
and satisfies the additional properties
characterizing $\scS_q$.
\end{enumerate}

The condition that $X_0$ is a 3-fold
is used in the proof that vanishing cycles
split-generate the Fukaya category,
cf.~Remarks \ref{rm:negativity} and
\ref{rm:generation}.
Sheridan \cite{Sheridan_CY} proved homological mirror symmetry
for Calabi-Yau hypersurfaces in projective spaces
along the lines of \cite{Sheridan_pants}.
In contrast to Sheridan's approach,
our proof is based on the relation
between Sheridan's immersed Lagrangian sphere
in a pair of pants and
vanishing cycles on Calabi-Yau hypersurfaces,
and gives the compatibility
of homological mirror symmetry
for the projective space
and its Calabi-Yau hypersurface
as in Remark \ref{rm:compatibility}.

This paper is organized as follows:
Sections \ref{sc:coherent_sheaves} and \ref{sc:fuk}
have little claim in originality,
and we include them for the readers' convenience.
In Section \ref{sc:coherent_sheaves},
we recall the description of the derived category
of coherent sheaves on $Z_q^*$
due to Seidel \cite{Seidel_K3}.
In Section \ref{sc:fuk},
we extend Seidel's discussion
on the Fukaya category of the quartic surface
to general projective Calabi-Yau hypersurfaces.
Strictly speaking,
the work of Fukaya, Oh, Ohta and Ono
\cite{Fukaya-Oh-Ohta-Ono}
that we rely on in this section
gives not a full-fledged $A_\infty$-category
but an $A_\infty$-algebra
for a Lagrangian submanifold
and an $A_\infty$-bimodule
for a pair of Lagrangian submanifolds.
While there is apparently no essential difficulty
in generalizing their work to construct an $A_\infty$-category
(for transversally intersecting sequence of Lagrangian submanifolds,
one can regard it as a single immersed Lagrangian submanifold
and use the work of Akaho and Joyce \cite{Akaho-Joyce}),
we do not attempt to settle this foundational issue in this paper.
Sections \ref{sc:negativity} and \ref{sc:sheridan_vc} are
at the heart of this paper.
In Section \ref{sc:negativity},
we prove the negativity of the large complex structure limit
monodromy using ideas of Seidel \cite{Seidel_K3} and Ruan \cite{Ruan_II}.
In Section \ref{sc:sheridan_vc},
we use ideas from
\cite{Seidel_GKQ} and \cite{Futaki-Ueda_Pn}
to reduce Floer cohomology computations
on vanishing cycles
needed in Section \ref{sc:fuk}
to a result of Sheridan \cite{Sheridan_pants}.


{\bf Acknowledgment}:
We thank Akira Ishii, Takeo Nishinou and
Nick Sheridan for valuable discussions.
We also thank the anonymous referees
for helpful suggestions and comments.
Y.~N. is supported
by Grant-in-Aid for Young Scientists (No.19740025).
K.~U. is supported
by Grant-in-Aid for Young Scientists (No.20740037).

\section{Derived category of coherent sheaves}
 \label{sc:coherent_sheaves}

Let $V$ be an $(n+2)$-dimensional complex vector space
spanned by $\{ v_i \}_{i=1}^{n+2}$,
and
$\{ y_i \}_{i=1}^{n+2}$ be the dual basis of $V^\vee$.
The projective space $\bP(V)$ has a full exceptional collection
$
 (F_k = \Omega_{\bP(V)}^{n+2-k}(n+2-k)[n+2-k])_{k=1}^{n+2}
$
by Beilinson \cite{Beilinson}.
The full dg subcategory of 
(the dg enhancement of) $D^b \coh \bP(V)$
consisting of $(F_k)_{k=1}^{n+2}$ is quasi-isomorphic
to the $\bZ$-graded category
$
 \dirC_{n+2}
$
with $(n+2)$ objects
$
 X_1, \ldots, X_{n+2}
$
and morphisms
$$
 \Hom_{\dirC_{n+2}} (X_j, X_k)
  =
\begin{cases}
 \Lambda^{k-j} V & j \le k, \\
 0 & \text{otherwise}.
\end{cases}
$$
The differential is trivial,
the composition is given by the wedge product,
and the grading is such that $V$ is homogeneous of degree one.
One can equip $(F_k)_{k=1}^{n+2}$
with a $\GL(V)$-linearization
so that this quasi-isomorphism
is $\GL(V)$-equivariant.
Let
$
 \iota_0 : Y_0
  \hookrightarrow \bP(V)
$
be the inclusion
of the union of coordinate hyperplanes
and set
$
 E_{0, k} = \iota_0^* F_k.
$
The total morphism dg algebra
$
 \bigoplus_{i, j=1}^{n+2} \hom (E_{0,i}, E_{0,j})
$
of this collection will be denoted by $\Sa$.

Let $C_{n+2}$ be the trivial extension category
of $\dirC_{n+2}$ of degree $n$
as defined in \cite[Section (10a)]{Seidel_K3}.
It is a category with the same object as $\dirC_{n+2}$.
The morphisms are given by
$$
 \Hom_{C_{n+2}}(X_j, X_k) = \Hom_{\dirC_{n+2}}(X_j, X_k)
  \oplus \Hom_{\dirC_{n+2}}(X_k, X_j)^\vee[-n],
$$
and the compositions are given by
$$
 (a, a^\vee) (b, b^\vee)
  = (a b, a^\vee (b \, \cdot \,))
   + (-1)^{\deg(a) (\deg(b) + \deg(b^\vee))} b^\vee(\, \cdot \, a).
$$
From this definition,
one can easily see that
$$
 \Hom_{C_{n+2}} (X_j, X_k)
  =
\begin{cases}
 \Lambda^{k-j} V & j < k, \\
 \Lambda^0 V \oplus \Lambda^{n+2} V[2] & j = k, \\
 \Lambda^{k-j+n+2} V[2] & j > k. \\
\end{cases}
$$
The total morphism algebra $Q_{n+2}$
of this category $C_{n+2}$ admits the following description:
Set
$
 \gamma = \zeta_{n+2} \id_{V}
$
for
$
 \zeta_{n+2} = \exp(2 \pi \sqrt{-1} / (n+2))
$
and let
$
 \Gamma_{n+2}
  = \la \gamma \ra \subset \SL(V)
$
be a cyclic subgroup of order $n+2$.
The group algebra $R_{n+2} = \bC \Gamma_{n+2}$ is
a semisimple algebra of dimension $n+2$,
whose primitive idempotents are given by
$$
 e_j = \frac{1}{n+2}
 (e + \zeta_{n+2}^{-j} \gamma + \cdots
    + \zeta_{n+2}^{-(n+1)j} \gamma^{n+1})
 \in \bC \Gamma_{n+2}.
$$
Let
$
 \Lambda V = \bigoplus_{i=0}^{n+2} \Lambda^i V
$
be the exterior algebra
equipped with the natural $\bZ$-grading
and
$
 \Qtilde_{n+2} = \Lambda V \rtimes \Gamma_{n+2}
$
be the semidirect product.
%
There is an $R_{n+2}$-algebra isomorphism
between $\Qtilde_{n+2}$ and $Q_{n+2}$
sending $e_k \Qtilde_{n+2} e_j$ to $\Hom_{C_{n+2}}(X_j, X_k)$.
This isomorphism does not preserve the $\bZ$-grading;
$Q_{n+2}$ is obtained from $\Qtilde_{n+2}$ by assigning degree
$\frac{n}{n+2} k$ to $\Lambda^k V \otimes \bC \Gamma_{n+2}$
and adding $\frac{2}{{n+2}} (k - j)$ to the piece
$
 e_k \Qtilde e_j.
$

Let $H$ be a maximal torus of $\SL(V)$
and $T$ be its image in $\PSL(V) = \SL(V) / \Gamma_{n+2}$.
The group $T$ acts on $Q_{n+2}$
by an automorphism of a graded $R_{n+2}$-algebra
in such a way that $[\diag(t_1, t_2, \dots, t_{n+2})]$
sends $v \otimes e_i \in e_{i+1} Q_{n+2} e_i$
to $(\diag(1, t_2 / t_1, \dots, t_{n+2} / t_1) \cdot v) \otimes e_i$.

The dg algebra $\Sa$ is characterized
by the following properties:

\begin{lemma}[{cf.~\cite[Lemma 10.2]{Seidel_K3}}]
Assume that a $T$-equivariant $A_\infty$-algebra $\scQ_{n+2}$
over $R_{n+2}$
satisfies the following properties:
\begin{itemize}
 \item
The cohomology algebra $H^*(\scQ_{n+2})$
is $T$-equivariantly isomorphic to $Q_{n+2}$
as an $R_{n+2}$-algebra.
 \item
$\scQ_{n+2}$ is not quasi-isomorphic
to $Q_{n+2}$.
\end{itemize}
Then one has a $R_{n+2}$-linear,
$T$-equivariant quasi-isomorphism
$
 \scQ_{n+2} \simto \scS_{n+2}.
$
\end{lemma}

\begin{proof}[Sketch of proof]
The proof of the fact
that these properties
are satisfied by $\scS_{n+2}$
is identical to
\cite[Section (10d)]{Seidel_K3}.
The uniqueness comes from the
Hochschild cohomology computations
in \cite[Section (10a)]{Seidel_K3}:
The Hochschild cohomology of $\Qtilde_{n+2}$ is given by
$$
 HH^{s+t}(\Qtilde_{n+2}, \Qtilde_{n+2})^t
  \cong \bigoplus_{\gamma \in \Gamma_{n+2}}
   \lb S^s(V^\gamma)^\vee
  \otimes \Lambda^{s+t - \codim V^\gamma} (V^\gamma)
  \otimes \Lambda^{\codim V^\gamma} (V / V^\gamma)
   \rb^{\Gamma_{n+2}},
$$
where
$
 S V = \bigoplus_{i=0}^\infty S^i V
$
is the symmetric algebra of $V$
(cf.~\cite[Proposition 4.2]{Seidel_K3}).
By the change of the grading
from $\Qtilde_{n+2}$ to $Q_{n+2}$,
one obtains
$$
 HH^{s+t}(Q_{n+2}, Q_{n+2})^{t}
  \cong \bigoplus_{\gamma \in \Gamma_{n+2}}
   \lb S^s(V^\gamma)^\vee
  \otimes \Lambda^{s + \frac{n+2}{n} t - \codim V^\gamma} (V^\gamma)
  \otimes \Lambda^{\codim V^\gamma} (V / V^\gamma)
   \rb^{\Gamma_{n+2}}.
$$
By passing to the $T$-invariant part,
one obtains
\begin{align}
 (HH^2(Q_{n+2}, Q_{n+2})^{2-d})^T
  &= (S^d V^\vee \otimes \Lambda^{n+2-d} V)^H \nonumber \\
  &=
\begin{cases}
 \bC \cdot y_1 \cdots y_{n+2} & d = n + 2, \\
 0 & \text{for all other $d > 2$},
\end{cases}
 \label{eq:hh2}
\end{align}
so that $\Sa$ is determined
by the above properties
up to quasi-isomorphism
\cite[Lemma 3.2]{Seidel_K3}.
\end{proof}

Let $\bP_{\Lambda_{\bN}} = \bP(V \otimes_\bC \Lambda_\bN)$
be the projective space over $\Lambda_\bN$
and $Y_q$ be the hypersurface defined by
$
 q (y_1^{n+2} + \cdots + y_{n+2}^{n+2}) + y_1 \cdots y_{n+2}
  = 0.
$
The geometric generic fiber of the family
$Y_q \to \Spec \Lambda_{\bN}$
is the smooth Calabi-Yau variety
$
 Y_q^* = Y_q \times_{\Lambda_{\bN}} \Lambda_\bQ
$
appearing in Introduction,
and the special fiber is $Y_0$ above.
The collection $E_{0, k}$ is the restriction
of the collection $E_{q, k}$ on $Y_q$
obtained from the Beilinson collection
on $\bP_{\Lambda_\bN}$,
and its restriction to $Y_q^*$
split-generates $D^b \coh Y_q^*$
by \cite[Lemma 5.4]{Seidel_K3}.

Let
$
 \Gamma
$
be the abelian subgroup of $\PSL_{n+2}(\bC)$
defined in \eqref{eq:Gamma}.
Each $E_{q, k}$ admits $(n+2)^n$ ways
of $\Gamma$-linearizations,
so that one obtains $(n+2)^{n+1}$ objects
of $D^b \coh Z_q = D^b \coh^\Gamma Y_q$,
whose total morphism dg algebra will be denoted by
$\scS_q$.
It is clear that their restriction to $Z_q^*$
split-generates $D^b \coh Z_q^*$, so that
one has the following:

\begin{lemma} \label{lm:coh_generation}
There is an equivalence
$$
 D^b \coh Z_q^* \cong D^\pi \scS_q^*
$$
of triangulated categories,
where
$
 \scS_q^* = \scS_q \otimes_{\Lambda_\bN} \Lambda_\bQ.
$\end{lemma}

We write the inverse image of $\Gamma \subset \PSL(V)$
by the projection $\SL(V) \to \PSL(V)$ as $\Gammatilde$,
and set
$
 Q
  = Q_{n+2} \rtimes \Gamma
  = \Lambda V \rtimes \Gammatilde.
$
Then the cohomology algebra of $\Sb_q$ is given by
$
 Q \otimes \Lambda_\bN,
$
and the central fiber is
$
 \scS_0 = \scS_{n+2} \rtimes \Gamma.
$
As explained in \cite[Section 3]{Seidel_K3},
first order deformations of the dg (or $A_\infty$-)algebra
$\scS_0$
are parametrized
by the {\em truncated Hochschild cohomology}
$HH^2(\scS_0, \scS_0)^{\le 0}$.

\begin{lemma}[{cf. \cite[Lemma 10.5]{Seidel_K3}}]
The truncated Hochschild cohomology of $\scS_0$ satisfies
$$
 HH^1(\scS_0,\scS_0)^{\le 0}
  = \bC^{n+1}, \quad
 HH^2(\scS_0,\scS_0)^{\le 0}
  = \bC^{2 n + 3}.
$$
\end{lemma}

\begin{proof}[Sketch of proof]
There is a spectral sequence
leading to $HH^*(\scS_0, \scS_0)^{\le 0}$
such that
$$
 E_2^{s, t} =
\begin{cases}
 HH^{s+t}(Q, Q)^t & t \le 0, \\
 0 & \text{otherwise}.
\end{cases}
$$
The isomorphism
\begin{align*}
 HH^{s+t}(Q, Q)^{t}
  \cong \bigoplus_{\gamma \in \Gammatilde}
   \lb S^s(V^\gamma)^\vee
  \otimes \Lambda^{s + \frac{n+2}{n} t - \codim V^\gamma} (V^\gamma)
  \otimes \Lambda^{\codim V^\gamma} (V / V^\gamma)
   \rb^{\Gammatilde}
\end{align*}
implies that $E_2^{s, t} = 0$
for $s < 0$ or $s + \frac{n+2}{n} t < 0$,
which ensures the convergence of the spectral sequence.
One can easily see
that $E_2^{s, t}$ for $s + t \le 2$
is non-zero only if
$$
 (s, t) = (0, 0), (1, 0), (2, 0), \text{ or } (n+2, -n).
$$
The first nonzero differential is $\delta_{n+1}$,
which is the Schouten bracket
with the order $n+2$ deformation class
$y_1 \cdots y_{n+2}$
from \eqref{eq:hh2}.
In total degree $s + t = 1$,
we have the $\Gammatilde$-invariant part of $V^\vee \otimes V$,
which is spanned by elements $y_k \otimes v_k$
satisfying
$$
 \delta_{n+1}^{1, 0}(y_k \otimes v_k) = y_1 \cdots y_{n+2}
$$
for $k = 1, \dots, n+2$.
In total degree $s + t = 2$, we have
\begin{itemize}
 \item
$(S^2 V^\vee \otimes \Lambda^2 V)^{\Gammatilde}$
generated by $(n+2)(n+1)/2$ elements $y_j y_k \otimes v_j \wedge v_k$
satisfying
$$
 \delta_{n+1}^{2, 0}(y_j y_k \otimes v_j \wedge v_k)
  = (y_1 \cdots y_{n+2}) y_k \otimes v_k
      - (y_1 \cdots y_{n+2}) y_j \otimes v_j,
$$
 \item
$(S^{n+2} V^\vee)^{\Gammatilde}$ spanned by
$y_k^{n+2}$ together with $y_1 \cdots y_{n+2}$.
\end{itemize}
The kernel of $\delta_{n+1}^{1, 0}$ is spanned by
$$
 y_1 \otimes v_1 - y_2 \otimes v_2
$$
and its $n+1$ cyclic permutations,
which sum up to zero.
The image of $\delta_{n+1}^{1,0}$ is spanned
by $y_1 \cdots y_{n+2}$.
The kernel of $\delta_{n+1}^{2, 0}$ is spanned by
$$
 y_1 y_2 \otimes v_1 \wedge v_2
  + y_2 y_3 \otimes v_2 \wedge v_3
  - y_1 y_3 \otimes v_1 \wedge v_3
$$
and its $n+1$ cyclic permutations,
which also sum up to zero.
Differentials $\delta_k^{s, t}$
for $k > n + 1$ and $s + t \le 2$ vanish,
and one obtains the desired result.
%
\end{proof}

Unfortunately,
the second truncated Hochschild cohomology group
$ HH^2(\scS_0,\scS_0)^{\le 0}$
has multiple dimensions,
so that one needs additional structures
to characterize $\scS_q$ as a deformation of $\scS_0$.
The strategy adopted by Seidel
is to use a $\bZ / (n+2) \bZ$-action
coming from the cyclic permutation of the basis of $V$:
Let $U_{n+2}$ be an automorphism of
$Q_{n+2} = \Lambda V \rtimes \Gamma_{n+2}$
as an $R_{n+2}$-algebra,
which acts on the basis of $V$ as
$v_k \mapsto v_{k+1}$.
This lifts to a $\bZ / (n + 2) \bZ$-action on
$\scS_0 = \scS_{n+2} \rtimes \Gamma$,
and $\scS_q$ is characterized as follows:

\begin{proposition}[{cf.~\cite[Proposition 10.8]{Seidel_K3}}]
 \label{pr:characterize_S2}
Let $\scQ_q$ be a one-parameter deformation
of $\scS_0 = \scS_{n+2} \rtimes \Gamma$,
which is
\begin{itemize}
 \item
$\bZ / (n+2) \bZ$-equivariant, and
 \item
non-trivial at first order.
\end{itemize}
Then $\scQ_q$ is quasi-isomorphic to $\psi^* \scS_q$
for some $\psi \in \End(\Lambda_\bN)^\times$.
\end{proposition}

The proof that these conditions
characterize $\scS_q$
comes from the fact that
the invariant part
of the second truncated Hochschild cohomology
of the central fiber $\scS_0$
with respect to the cyclic group action
induced by $\scU_0$ is
one-dimensional
\cite[Lemma 10.7]{Seidel_K3};
$$
 HH^2(\scS_0, \scS_0)^{\le 0, \, \bZ / (n+2) \bZ}
  \cong \bC \cdot (y_1^{n+2}+\dots+y_{n+2}^{n+2}).
$$
The proof that these conditions
are satisfied by $\scS_q$
carries over verbatim
from \cite[Section (10d)]{Seidel_K3}.

\section{Fukaya categories}
 \label{sc:fuk}

Let
$
 X
  = \Proj \bC[x_1, \ldots, x_{n+2}]
$
be an $(n+1)$-dimensional complex projective space
and $o_X$ be the anticanonical bundle on $X$.
Let further $h$ be a Hermitian metric on $o_X$
such that the compatible unitary connection $\nabla$
has the curvature $- 2 \pi \sqrt{-1} \omega_X$,
where $\omega_X$ is $n + 2$ times
the Fubini-Study K\"{a}hler form on $X$.
Any complex submanifold of $X$ has a symplectic structure
given by the restriction of $\omega_X$.
The restriction of $(o_X, \nabla)$
to any Lagrangian submanifold $L$
has a vanishing curvature,
and $L$ is said to be {\em rational}
if the monodromy group of this flat connection is finite.
Note that this condition is equivalent to the existence
of a flat multi-section $\lambda_L$ of $o_X|_L$
which is of unit length everywhere.

Two sections
$
 \sigma_{X, \infty} = x_1 \cdots x_{n+2}
$
and
$
 \sigma_{X, 0} = x_1^{n+2} + \cdots + x_{n+2}^{n+2}
$
of $o_X$ generate a pencil
$
 \{ X_z \}_{z \in \bP_\bC^1}
$
of hypersurfaces
$$
 X_z = \{ x \in X \mid
  \sigma_{X,0}(x) + z \sigma_{X, \infty}(x) = 0 \},
$$
such that $X_0$ is the Fermat hypersurface and
$X_\infty$ is the union of $n+2$ coordinate hyperplanes.
The complement $M = X \setminus X_\infty$
is the big torus of $X$,
which can naturally be identified as
\begin{align*}
 M
  &= \{ x \in \bC^{n+2} \mid x_1 \cdots x_{n+2} \ne 0 \}
  / \bCx
  \cong \{ x \in \bC^{n+2} \mid x_1 \cdots x_{n+2} = 1 \}
  / \Gammaa^*,
\end{align*}
where
$
 \Gammaa^* = \{ \zeta \id_{\bC^{n+2}} \mid \zeta^{n+2} = 1 \}
$
is the kernel
of the natural projection
from $\SL_{n+2}(\bC)$ to $\PSL_{n+2}(\bC)$.
The map
$$
 \pi_M = \sigma_{X,0} / \sigma_{X,\infty} : M \to \bC
$$
is a Lefschetz fibration,
which has
$n + 2$ groups
of $(n+2)^n$ critical points
with identical critical values.
The group $\Gammab^* = \Hom(\Gammab, \bCx)$
of characters of the group $\Gammab$
defined in \eqref{eq:Gamma}
acts freely on $M$
through a non-canonical isomorphism
$\Gammab^* \cong \Gammab$
and the natural action of $\Gammab \subset \PSL_{n+2}(\bC)$
on 
$X$.
The quotient
$$
 \Mbar
  = M / \Gammab^*
  = \{ u = (u_1, \ldots, u_{n+2}) \in \bC^{n+2}
         \mid u_1 \cdots u_{n+2} = 1 \}
$$
is another algebraic torus,
where the natural projection $M \to \Mbar$ is given by
$
   u_k = x_k^{n+2}.
$
The map $\pi_M$ is $\Gammab^*$-invariant and
descends to the map
$
 \pi_\Mbar(u) = u_1 + \cdots + u_{n+2}
$
from the quotient:
\begin{equation*}
\begin{psmatrix}[rowsep=1cm]
 M & & \bC \\
  & \Mbar &
\end{psmatrix}
\psset{shortput=nab,arrows=->,nodesep=3pt,labelsep=3pt}
\small
\ncline{1,1}{1,3}^{\pi_M}
\ncline{1,1}{2,2}
\ncline{2,2}{1,3}_{\pi_\Mbar}
\end{equation*}
The map
$
 \pi_\Mbar : \Mbar \to \bC
$
is the Landau-Ginzburg potential
for the mirror of $\bP^{n+1}$,
which has $n + 2$ critical points
with critical values
$
 \{ (n+2) \zeta_{n+2}^{-i} \}_{i=1}^{n+2}
$
where
$
 \zeta_{n+2} = \exp[2 \pi \sqrt{-1} / (n+2)].
$
Choose the origin as the base point
and take the distinguished set
$(\delta_i)_{i=1}^{n+2}$
of vanishing paths
$
 \delta_i : [0, 1] \ni t
  \mapsto (n+2) \zeta_{n+2}^{-i} \, t \in \bC
$
as in Figure \ref{fg:Pn_vp}.
The corresponding vanishing cycles in
$
 \Mbar_0 = \pi_\Mbar^{-1}(0)
$
will be denoted by $V_i$.

\begin{figure}
\centering
\input{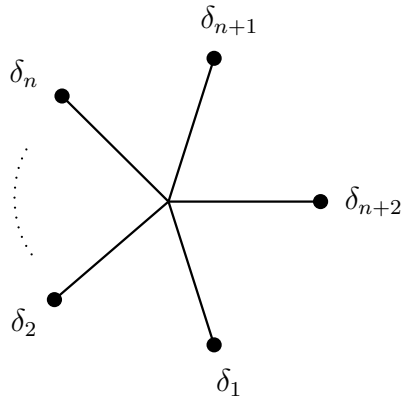}
\caption{The distinguished set $(\delta_i)_{i=1}^{n+2}$
of vanishing paths}
\label{fg:Pn_vp}
\end{figure}

Let $\Fa$ be the $A_\infty$-category
whose set of objects is $\{ V_i \}_{i=1}^{n+2}$ and
whose spaces of morphisms are
Lagrangian intersection Floer complexes.
This is a full $A_\infty$-subcategory of the Fukaya category
$\scF(\Mbar_0)$ of the exact symplectic manifold $\Mbar_0$.
See \cite{Seidel_PL}
for the Fukaya category
of an exact symplectic manifold,
and \cite{Fukaya-Oh-Ohta-Ono}
for that of a general symplectic manifold.
We often regard the $A_\infty$-category
$\Fa$ with $n+2$ objects
as an $A_\infty$-algebra
over the semisimple ring $R_{n+2}$
of dimension $n+2$.

As explained in Section \ref{sc:sheridan_vc} below,
the affine variety $\Mbar_0$ is an $(n+2)$-fold cover
of the $n$-dimensional pair of pants $\scP^n$,
and contains $n + 2$ Lagrangian spheres
$\{ L_i \}_{i=1}^{n+2}$
whose projection to $\scP^n$ is the Lagrangian immersion
studied by Sheridan \cite{Sheridan_pants}.
Let $\scA_{n+2}$ be the full $A_\infty$-subcategory
of $\scF(\Mbar_0)$ consisting of these Lagrangian spheres.
The following proposition
is proved in Section \ref{sc:sheridan_vc}:

\begin{proposition}
 \label{prop:sheridan_vc}
The Lagrangian submanifolds $L_i$ and $V_i$ are isomorphic
in $\scF(\Mbar_0)$.
\end{proposition}

The inclusion $\Mbar_0 \subset \Mbar$
induces an isomorphism
$
 \pi_1(\Mbar_0)
  \cong \pi_1(\Mbar)
$
of the fundamental group.
Let $T$ be the torus
dual to $\Mbar$
so that
$
 \pi_1(\Mbar)
  \cong T^*
  := \Hom(T, \bCx).
$
One can equip $\Fa$
with a $T$-action
by choosing lifts of $V_i$
to the universal cover of $\Mbar_0$.
%
Let $\scF_0$ be the Fukaya category of $M_0$
consisting of $N = (n+2)^{n+1}$ vanishing cycles
$\{ \Vtilde_i \}_{i=1}^N$ of $\pi_M$
obtained by pulling-back $\{ V_i \}_{i=1}^{n+2}$.
The covering $M_0 \to \Mbar_0$
comes from a surjective group homomorphism
$\pi_1(\Mbar_0) \to \Gammab^*$,
which induces an inclusion
$\Gammab \hookrightarrow T$ of the dual group.
It follows from \cite[Equation (8.13)]{Seidel_K3}
that $\scF_0$ is quasi-isomorphic to $\scF_{n+2} \rtimes \Gammab$,
which in turn is quasi-isomorphic to $\scA_{n+2} \rtimes \Gammab$
by Proposition \ref{prop:sheridan_vc}.

The following proposition is due to Sheridan:

\begin{proposition}[{\cite[Proposition 5.15]{Sheridan_pants}}]
 \label{pr:sheridan}
$\scA_{n+2}$ is $T$-equivariantly quasi-isomorphic to $\scS_{n+2}$.\end{proposition}

Since $\scS_0 = \scS_{n+2} \rtimes \Gamma$,
one obtains the following:

\begin{corollary} \label{cr:affine_hms}
$\scF_0$ is quasi-isomorphic to $\scS_0$.
\end{corollary}

The vanishing cycles $\{ \Vtilde_i \}_{i=1}^N$ are
Lagrangian submanifolds
of the projective Calabi-Yau manifold $X_0$,
which are rational since they are contractible in $M$.
To show that they split-generate the Fukaya category
of $X_0$,
Seidel introduced the notion of {\em negativity}
of a graded symplectic automorphism.
Let $\frakL_{X_0} \to X_0$ be the bundle of
unoriented Lagrangian Grassmannians
on the projective Calabi-Yau manifold $X_0$.
The {\em phase function}
$
 \alpha_{X_0} : \frakL_{X_0} \to S^1
$
is defined by
$$
 \alpha_{X_0}(\Lambda) =
  \frac{\eta_{X_0} (e_1 \wedge \cdots \wedge e_n)^2}
       {|\eta_{X_0} (e_1 \wedge \cdots \wedge e_n)|^2},
$$
where
$
 \Lambda = \vspan_{\bR} \{e_1, \ldots, e_n \}
  \in \frakL_{X_0, x}
$
is a Lagrangian subspace of $T_x X_0$ and
$\eta_{X_0}$ is a holomorphic volume form on $X_0$.
The {\em phase function}
$
 \alpha_\phi : \frakL_{X_0} \to S^1
$
of a symplectic automorphism
$
 \phi : X_0 \to X_0
$
is defined by sending $\Lambda \in \frakL_{X_0, x}$ to
$
 \alpha_\phi(\Lambda) =
  \alpha_{X_0}(\phi_*(\Lambda)) / \alpha_{X_0}(\Lambda),
$
and a {\em graded symplectic automorphism}
is a pair $\phitilde = (\phi, \alphatilde_\phi)$
of a symplectic automorphism $\phi$ and
a lift $\alphatilde_\phi : \frakL_{X_0} \to \bR$
of the phase function $\alpha_\phi$
to the universal cover $\bR$ of $S^1$.
The group of graded symplectic automorphisms of $X_0$
will be denoted by $\Auttilde(X_0)$.
A graded symplectic automorphism
$\phitilde \in \Auttilde(X_0)$
is {\em negative}
if there is a positive integer $d_0$
such that
$
 \alphatilde_{\phi^{d_0}}(\Lambda) < 0
$
for all $\Lambda \in \frakL_{X_0}$.

The {\em phase function}
$
 \alpha_L : L \to S^1
$
of a Lagrangian submanifold $L \subset X_0$
is defined similarly by
$
 \alpha_L(x) = \alpha_{X_0}(T_x L),
$
and a {\em grading} of $L$ is a lift
$\alphatilde_L : L \to \bR$ of $\alpha_L$
to the universal cover of $S^1$.
Let $\Lambda_0$ be the local subring of $\Lambda_\bQ$
containing only non-negative powers of $q$,
and $\Lambda_+$ be the maximal ideal of $\Lambda_0$.
For a quintuple
$
 L^\sharp = (L, \alphatilde_L, \$_L, \lambda_L, J_L)
$
consisting of a rational Lagrangian submanifold $L$,
a grading $\alphatilde_L$ on $L$,
a spin structure $\$_L$ on $L$,
a multi-section $\lambda_L$ of $o_{X_0}|_L$,
and a compatible almost complex structure $J_L$,
one can endow the cohomology group
$
H^*(L; \Lambda_0)
$
with the structure $\{ \frakm_k \}_{k=0}^\infty$
of a {\em filtered $A_\infty$-algebra}
\cite[Definition 3.2.20]{Fukaya-Oh-Ohta-Ono},
which is well-defined
up to isomorphism
\cite[Theorem A]{Fukaya-Oh-Ohta-Ono}.
The map $\frakm_0 : \Lambda_0 \to H^1(L; \Lambda_0)$
comes from holomorphic disks bounded by $L$,
and measures the {\em anomaly} or {\em obstruction}
to the definition of Floer cohomology.
A solution $b \in H^1(L; \Lambda_+)$
to the {\em Maurer-Cartan equation}
$$
 \sum_{k=0}^\infty \frakm_k(b, \cdots, b) = 0
$$
is called a {\em bounding cochain}.
A {\em rational Lagrangian brane} is a pair
$\Ld = (L^\sharp, b)$
of $L^\sharp$
and a bounding cochain $b \in H^1(L; \Lambda_+)$.
For a pair $\Ld_1 = (L_1^\sharp, b_1)$
and $\Ld_2 = (L_2^\sharp, b_2)$
of rational Lagrangian branes,
the {\em Floer cohomology}
$
 HF(\Ld_1, \Ld_2; \Lambda_0)
$
is well-defined up to isomorphism.
The {\em Fukaya category} $\Fuk(X_0)$
is an $A_\infty$-category over $\Lambda_\bQ$
whose objects are rational Lagrangian branes
and whose spaces of morphisms are
Lagrangian intersection Floer complexes.

Let $\scF_q$ be the full $A_\infty$-subcategory
of $\scF(X_0)$
consisting of vanishing cycles $\Vtilde_i$
equipped with the trivial complex line bundles,
the canonical gradings and zero bounding cochains.
Since the restrictions of $(o_X, \nabla)$ to vanishing cycles
are trivial flat bundles,
the category $\scF_q$ is defined over $\Lambda_\bN$.

Let $\eta_M$ be the unique up to scalar
holomorphic volume form on $M$
which extends to a rational form on $X$
with a simple pole along $X_\infty$.
This gives a holomorphic volume form
$\eta_M / d z$ on each fiber $M_z = \pi_M^{-1}(z)$,
so that $\pi_M : M \to \bC$ is a locally trivial
fibration of graded symplectic manifolds
outside the critical values.
Let $\gamma_\infty : [0, 2 \pi] \to \bC$
be a circle of large radius $R \gg 0$ and
$
 \htilde_{\gamma_\infty} \in \Auttilde(M_R)
$
be the monodromy along $\gamma_\infty$.
Since $\gamma_\infty$ is homotopic
to a product of paths around each critical values,
one sees that $\htilde_{\gamma_\infty}$ is isotopic
to a composition of Dehn twists
along vanishing cycles.
We prove the following in Section \ref{sc:negativity}:

\begin{proposition}[cf. {\cite[Proposition 7.22]{Seidel_K3}}]
 \label{prop:negativity}
The graded symplectic automorphism
$
 \htilde_{\gamma_\infty} \in \Auttilde(M_R)
$
is isotopic to a graded symplectic automorphism
$
 \phitilde \in \Auttilde(M_R)
$
whose extension to $X_R$ has the following property:
There is an arbitrary small neighborhood
$W \subset X_R$ of the subset
$
 \Sing(X_{\infty}) \cap X_R
$ 
such that $\phi(W) = W$ and
$\phitilde|_{X_R \setminus W}$ is negative.
\end{proposition}

Here $\Sing(X_\infty)$ is the singular locus of $X_\infty$,
which is the union of $(n - 1)$-dimensional projective spaces.

\begin{lemma}[{cf.~\cite[Lemma 9.2]{Seidel_K3}}]
 \label{lm:Fuk_generation}
If $n = 3$,
then any rational Lagrangian brane
is contained in split-closed derived category of
$\scF_q^* = \scF_q \otimes_{\Lambda_\bN} \Lambda_\bQ$;
$$
 D^\pi \scF (X_0) \cong D^\pi \scF_q^*.
$$
\end{lemma}

The proof is identical to that
of \cite[Lemma 9.2]{Seidel_K3},
which is based on Seidel's long exact sequence
\cite{Seidel_LES}
(cf.~also \cite[Section (9c)]{Seidel_K3}
and \cite{Oh_SLES}).

\begin{remark}[{cf.~\cite[Remark 9.3]{Seidel_K3}}]
 \label{rm:negativity}
If $n = 3$,
then the real dimension of $\Sing(X_\infty) \cap X_0$ is two,
so that any Lagrangian submanifold can be made disjoint
from a sufficiently small neighborhood $W$
of $\Sing(X_\infty) \cap X_0$
by a generic perturbation.
This is the only place where we use the condition $n = 3$,
and one can show the equivalence
\eqref{eq:hms} for any $n$
with $D^\pi \scF(X_0)$ replaced
by the split-closure of Lagrangian branes
which can be perturbed away from $\Sing(X_\infty) \cap X_0$.
\end{remark}

A notable feature of Floer cohomologies
over $\Lambda_0$
is their dependence
on Hamiltonian isotopy:
For a pair $(\Ls_0, \Ls_1)$
of Lagrangian submanifolds
equipped with auxiliary choices,
a symplectomorphism $\psi : X_0 \to X_0$
induces an isomorphism
$$
 \psi_* : (H^*(\Ls_i; \Lambda_0)) , \frakm_k)
  \to (H^*(\psi(\Ls_i); \Lambda_0)) , \frakm_k)
$$
of filtered $A_\infty$-algebras
\cite[Theorem A]{Fukaya-Oh-Ohta-Ono},
which induces a map $\psi_*$
on the set of bounding cochains
preserving the Floer cohomology over $\Lambda_0$
\cite[Theorem G.3]{Fukaya-Oh-Ohta-Ono}:
$$
 HF((\Ls_0, b_0), (\Ls_1, b_1); \Lambda_0)
  \cong HF((\psi(\Ls_0), \psi_*(b_0)),(\psi(\Ls_1), \psi_*(b_1))
   ;\Lambda_0).
$$
On the other hand,
if we move $\Ls_0$ and $\Ls_1$
by two distinct Hamiltonian isotopies
$\psi^0$ and $\psi^1$,
then the Floer cohomology over $\Lambda_\bQ$ is preserved
\cite[Theorem G.4]{Fukaya-Oh-Ohta-Ono}
$$
 HF((\Ls_0, b_0), (\Ls_1, b_1); \Lambda_\bQ)
  \cong HF((\psi^0(\Ls_0), \psi^0_*(b_0)),
   (\psi^1(\Ls_1), \psi^1_*(b_1))
   ;\Lambda_\bQ),
$$
whereas
the Floer cohomology over $\Lambda_0$ may not be preserved;
$$
 HF((\Ls_0, b_0), (\Ls_1, b_1); \Lambda_0)
  \not \cong HF((\psi^0(\Ls_0), \psi^0_*(b_0)),
   (\psi^1(\Ls_1), \psi^1_*(b_1))
   ;\Lambda_0).
$$
See \cite[Section 3.7.6]{Fukaya-Oh-Ohta-Ono}
for a simple example where this occurs.
This phenomenon is used
by Seidel \cite[Section (8g) and (11a)]{Seidel_K3}
to prove the following:

\begin{proposition}[{cf.~\cite[Proposition 11.1]{Seidel_K3}}]
 \label{pr:Fuk_non-trivial}
The $A_\infty$-algebra
$
 \Fqb \ofo
$
is not quasi-isomorphic
to the trivial deformation
$
 \scF_0 \otimes_\bC \fo.
$
\end{proposition}

To show this,
Seidel takes a rational Lagrangian submanifold $L_{1/2}$
in $X_z$ for sufficiently large $z$
as follows:
\begin{enumerate}
 \item
Consider a pencil
$\{ X_z \}_{z \in \bP_\bC^1}$
generated by two section
$\sigma_{X, \infty} =  x_1 \cdots x_{n+2}$
and
$\sigma_{X,0} = x_1^2 (x_2^2 + x_3^2) x_4 \cdots x_{n+1}$,
whose general fiber is singular.
Let $C = \{ x_{n+2} = 0 \}$ be an irreducible component
of $X_\infty = \{ x_1 \cdots x_{n+2} = 0 \} \subset X$,
and $C_\infty = C \cap X_\infty$ be the intersection
with other components.
If we write $C_0 = X_0 \cap C$,
then the set $C_0 \setminus C_\infty$
is the union
of two $(n-1)$-planes
$\{ x_2 = \pm \sqrt{-1} x_3 \}$.
 \item
Let
$
 K_{1/2} = \{ 2 |x_1| = |x_2| = \cdots = |x_{n+2}| \}
  \subset C \setminus C_\infty
$
be a Lagrangian $n$-torus in $C$,
which is a fiber of the moment map
for the torus action.
The intersection $K_{1/2} \cap C_0$
consists of two $(n-1)$-tori.
 \item
Take a Hamiltonian function $H$ on $C$
supported on a neighborhood of the two $(n-1)$-tori
such that the corresponding Hamiltonian vector field
points in opposite directions
transversally to two $(n-1)$-tori.
By flowing $K_{1/2}$ along the Hamiltonian vector field
in both negative and positive time directions,
one obtains a family $(K_r)_{r \in [0,1]}$
of Lagrangian submanifolds of $C \setminus C_\infty$.
 \item
The Lagrangian submanifolds $K_r$
for $r \ne 1/2$ are disjoint from $C_0$.
They are exact Lagrangian submanifolds
with respect to the one-form $\theta_{C \setminus C_0}$
obtained by pulling back the connection on $o_X$
via $\sigma_{X, 0}|_{C \setminus C_0}$.
 \item
Now perform a generic perturbation of $\sigma_{X, 0}$
so that a general member $X_z$
of the pencil is smooth.
One still has a Lagrangian submanifold
$K_{1/2} \subset C \setminus C_\infty$
satisfying the following:
\begin{itemize}
 \item
$K_{1/2} \cap C_0$ consists of two $(n-1)$-tori.
 \item
By flowing $K_{1/2}$ along a Hamiltonian vector field,
one obtains a family $(K_r)_{r \in [0,1]}$
of Lagrangian submanifolds of $C \setminus C_\infty$.
 \item
$K_r$ for $r \ne 1/2$ are disjoint from $C_0$.
They are exact Lagrangian submanifolds
of $C \setminus C_0$.
\end{itemize}
 \item
By parallel transport along the graph
$$
 \Xhat = \{ (y, x) \in \bC \times X \mid
  \sigma_{X, \infty}(x) = y \sigma_{X, 0}(x) \}
 \xto{\text{$y$-projection}} \bC
$$
of the pencil,
one obtains a Lagrangian torus $L_{1/2}$
in $X_z$
for sufficiently large $z = 1/y$,
satisfying the following conditions:
\begin{itemize}
 \item
The intersection $Z = L_{1/2} \cap X_{z, \infty}$
of $L_{1/2} \cong (S^1)^n$
with the divisor $X_{z, \infty} = X_z \cap X_\infty$ at infinity
is a smooth $(n-1)$-dimensional manifold
disjoint from $\Sing(X_{\infty}) \cap X_z$.
(In fact, it is a disjoint union of two $(n-1)$-tori;
$
 Z = \{ 1/4, 3/4 \} \times (S^1)^{n-1}.
$)
 \item
By flowing $L_{1/2}$ by a Hamiltonian vector field,
one obtains a family $(L_r)_{r \in [0,1]}$
of Lagrangian submanifolds of $X_z$.
 \item
$L_r$ for any $r \in [0,1]$ admits a grading.
 \item
$L_r$ for $r \ne 1/2$
are disjoint from $X_{z, \infty}$.
They are exact Lagrangian submanifolds
in the affine part $M_z = X_z \setminus X_{z, \infty}$ of $X_z$.
\end{itemize}
\end{enumerate}

If the perturbation of $\sigma_{X,0}$ is generic,
then there are no non-constant stable holomorphic disks
in $X_z$ bounded by $L_r$ for $r \in [0,1]$
with area less than $2$.
Indeed, such a disk cannot have a sphere component
since a holomorphic sphere has area at least $n+2$.
If a holomorphic disk exists in $X_z$ for all sufficiently large $z$,
then Gromov compactness theorem gives
a holomorphic disk in $X_\infty$ bounded by $K_r$.
This disk either have sphere components
in irreducible components of $X_\infty$ other than $C$,
or passes through $C_\infty \cap C_0$.
The former is impossible
since sphere components have area at least $n+2$,
and the latter is impossible
for a disk of area less than 2
since such disks have fixed intersection points
with $C_\infty$
by classification
\cite[Theorem 10.1]{Cho_CT}
of holomorphic disks in $C$
bounded by $K_r$.

The absence of holomorphic disks
of area less than 2
shows that the Lagrangian submanifolds
$\Ld_0 = (\Ls_0, 0)$ and $\Ld_1 = (\Ls_1, 0)$
equipped with auxiliary data
and the zero bounding cochains
give objects
of the first order Fukaya category
$
 D^\pi \scF_q \ofo.
$
Now the argument of Seidel
\cite[Section (8g)]{Seidel_K3} shows
the following:
\begin{enumerate}
 \item
The spaces
$H^0(\hom_{\scF_0}(\Ld_i, \Ld_j))$
are one-dimensional for $0 \le i \le j \le 1$.
 \item
The product
$$
 H^0(\hom_{\scF_0}(\Ld_1, \Ld_0))
  \otimes H^0(\hom_{\scF_0}(\Ld_0, \Ld_1))
  \to H^0(\hom_{\scF_0}(\Ld_0, \Ld_0))
$$
vanishes.
 \item
The map
$$
\begin{CD}
 H^0(\hom_{\scF_q}(\Ld_1, \Ld_0) \ofo)
  \otimes_\bC
 H^0(\hom_{\scF_q}(\Ld_0, \Ld_1) \ofo) \\
 @VVV \\
 H^0(\hom_{\scF_q}(\Ld_0, \Ld_0) \ofo)
\end{CD}
$$
induced by $\frakm_2^{\scF_q}$ is non-zero.
\end{enumerate}
The point is that $L_0$ and $L_1$ are
exact Lagrangian submanifolds of $M_z$,
which are not isomorphic in $\scF(M_z)$,
but are Hamiltonian isotopic in $X_z$,
so that they are isomorphic in
$
 D^\pi(\scF_q \otimes_{\Lambda_\bN} \Lambda_\bZ).
$
Now \cite[Lemma 3.9]{Seidel_K3} concludes
the proof of Proposition \ref{pr:Fuk_non-trivial}.

The symplectomorphism
$
 \phibar_0 : \Mbar_0 \to \Mbar_0
$
sending
$
 (u_1, \ldots, u_{n+2})
$
to
$
 (u_2, \ldots, u_{n+2}, u_1)
$
lifts to a $\bZ / (n+2)$-action on $\scF_q$
just as in \cite[Section (11b)]{Seidel_K3}.
It follows that $\scF_q$ satisfies
all the properties
characterizing $\scS_q$
in Proposition \ref{pr:characterize_S2},
and one obtains the following;

\begin{proposition} \label{pr:hms_projective}
$\scF_q$ is quasi-isomorphic to $\psi^* \scS_q$
for some $\psi \in \End(\Lambda_\bN)^\times$.
\end{proposition}

Theorem \ref{th:hms} follows
from Lemma \ref{lm:coh_generation},
Lemma \ref{lm:Fuk_generation},
and Proposition \ref{pr:hms_projective}.

\begin{remark} \label{rm:generation}
Since the Lagrangian torus used
in the proof of Proposition \ref{pr:Fuk_non-trivial}
does not intersect with $\Sing(X_\infty)$,
the proof of Proposition \ref{pr:Fuk_non-trivial}
(and hence Proposition \ref{pr:hms_projective})
works for any $n$.
Then the argument of Sheridan
\cite[Section 8.2]{Sheridan_CY},
based on a split-generation criterion
announced by Abouzaid, Fukaya, Oh, Ohta, and Ono,
shows that $\{ L_i \}_{i=1}^{n+2}$ split-generates
$D^\pi \scF(X_0)$
for any $n$.
\end{remark}

\section{Negativity of monodromy}
 \label{sc:negativity}

In this section,
we prove Proposition \ref{prop:negativity}
by using local models of the quasi-Lefschetz pencil $\{X_z\}$
along the lines of \cite[Section 7]{Seidel_K3}.
In the case where $\dim X_z \ge 3$, we need
\cite[Assumption 7.8]{Seidel_K3} and a generalization of
\cite[Assumption 7.5]{Seidel_K3}.

\begin{ass}[{\cite[Assumption 7.8]{Seidel_K3}}] \label{ass:local_model1}
Let $n \ge 2$ and $2 \le k \le n+1$.
\begin{itemize}
  \item $Y \subset \bC^{n+1} = \bR^{2n+2}$ is an open ball around the origin
        equipped with the standard symplectic form $\omega_Y$ and
        the $T^k$-action 
        \[
          \rho_s(y) = (e^{\ii s_1}y_1, \dots , e^{\ii s_k}y_k,
                       y_{k+1}, \dots , y_{n+1})
        \]
        with moment map $\mu : Y \to \bR^k$.
        For any regular value $r \in \bR^k$ of $\mu$,
        the symplectic reduction 
        $Y^{\red} = Y^{\red, r} = \mu^{-1}(r)/T^k$ 
        can be identified with an open subset in $\bC^{n+1-k}$ 
        equipped with the standard symplectic form.
  \item $J_Y$ is a complex structure on $Y$ which is tamed by $\omega_Y$.
        At the origin, it is $\omega_Y$-compatible and $T^k$-invariant.
  \item $p:Y \to \bC$ is a $J_Y$-holomorphic function with the
        following properties:
        \begin{enumerate}
        \renewcommand{\labelenumi}{$(\roman{enumi})$}
          \item $p( \rho_s(y) ) 
                 = e^{\ii(s_1  + \dots + s_k)} p(y)$.
          \item $\partial_{y_1} \dots \partial_{y_k}p$ is nonzero at $y=0$.
        \end{enumerate}
  \item $\eta_Y$ is a $J_Y$-complex volume form on 
        $Y \backslash p^{-1}(0)$ such that $p(y) \eta_Y$
        extends smoothly on $Y$, which is nonzero at $y=0$.
\end{itemize}
\end{ass}

In this situation, the monodromy $h_{\zeta}$ 
satisfy the following:

\begin{prop}[{\cite[Lemma 7.16]{Seidel_K3}}]
For every $d>0$ and $\epsilon >0$, there exists $\delta >0$ such that 
the following holds.
For every $y \in Y_{\zeta} = p^{-1}(\zeta)$ with $0< \zeta < \delta$ 
and $\|y\| <\delta$, and every Lagrangian subspace 
$\Lambda^v \subset T_yY_{\zeta}$, the $d$-fold monodromy $h_{\zeta}^d$
is well-defined near $y$, and satisfies
\[
    \tilde{\alpha}_{h_{\zeta}^d} (\Lambda^v) \le -2d 
    + n + 1 + \epsilon.
\]
\end{prop}

The other local model is the following:

\begin{ass} \label{ass:local_model}
Let $n \ge 2$ and $2 \le k \le n+1$.
\begin{itemize}
  \item $Y \subset \bC^{n+1} = \bR^{2n+2}$ is an open ball around the origin
        equipped with the standard symplectic form $\omega_Y$ and
        the $T^k$-action 
        \begin{equation}
          \rho_s(y) = (e^{\ii s_1}y_1, \dots , e^{\ii s_k}y_k,
                       y_{k+1}, \dots , y_{n+1})
        \label{eq:torus_action}
        \end{equation}
        with moment map $\mu : Y \to \bR^k$.
        For any regular value $r \in \bR^k$ of $\mu$,
        the symplectic reduction 
        $Y^{\red} = Y^{\red, r} = \mu^{-1}(r)/T^k$ 
        can be identified with an open subset in $\bC^{n+1-k}$ 
        equipped with the standard symplectic form.
  \item $J_Y$ is a complex structure on $Y$ which is tamed by $\omega_Y$.
        At the origin, it is $\omega_Y$-compatible and $T^k$-invariant.
  \item $p$ is a $J_Y$-meromorphic function on $Y$ satisfying
         the following two conditions:
        \begin{enumerate}
        \renewcommand{\labelenumi}{$(\roman{enumi})$}
          \item $p( \rho_s(y) ) 
                 = e^{\ii(-s_1 + s_2 + \dots + s_k)} p(y)$.
        \end{enumerate}
        This implies that $p$ can be written as
        \[
          p(y) = \frac{y_2 \dots y_k}{y_1} 
                 q(|y_1|^2/2, \dots , |y_k|^2/2, y_{k+1}, \dots, y_{n+1})
        \]
        for some $q$.
        \begin{enumerate}
        \setcounter{enumi}{1}
        \renewcommand{\labelenumi}{$(\roman{enumi})$}
          \item $q$ is a smooth function defined on $Y$, $q(0) = 1$, and 
          $q(y) \ne 0$ for any $y \in Y$.
        \end{enumerate}
  \item $\eta_Y$ is a $J_Y$-complex volume form on 
        $Y \backslash p^{-1}(0)$ such that $y_2 \dots y_k \eta_Y$
        extends smoothly on $Y$.
        It is normalized so that
        $y_2 \dots y_k \eta_Y = dy_1 \wedge \dots \wedge dy_{n+1}$ at $y=0$.
\end{itemize}
\end{ass}

In this setting, we will show the negativity of the monodromy 
in the following sense:

\begin{prop}[cf. {\cite[Lemma 7.16]{Seidel_K3}}] \label{Seidel7.16}
  For any $d > 0$ and $\epsilon > 0$, 
  there is $\delta_1 > \delta_2 > 0$ 
  such that for $\zeta \in \bC$ with $0<|\zeta| < \delta_1$ and
  $y \in Y_{\zeta}$ with $\|y\| < \delta_1$ and $|y_1| > \delta_2$,
  the $d$-fold monodromy $h_{\zeta}^d$ is well-defined, and
  \[
    \tilde{\alpha}_{h_{\zeta}^d} (\Lambda^v) \le -2d 
    \frac{1}{1+ |\zeta|^2/|y_3|^{2(k-1)}} + n+1 + \epsilon
  \]
  for all $\Lambda^v \in Y_{\zeta}$, provided
 $|y_2| \le |y_3| \le \dots \le |y_k|$.
\end{prop}

Note that
\[
  \frac{1}{1+ |\zeta|^2/|y_3|^{2(k-1)}}
\]
is uniformly bounded from above on the complement of 
a neighborhood of $y_2 = y_3 = 0$.

Let $J'_Y$ be the constant complex structure on $Y$
which coincides with $J_Y$ at the origin,
and $\eta'_Y$ be the constant $J'_Y$-complex volume form
given by
\[
  \eta'_Y = dy_1 \wedge \frac{dy_2}{y_2} \wedge \dots \wedge \frac{dy_k}{y_k}
            \wedge \eta'_{Y^{\red}} 
\]
for some $\eta'_{Y^{\red}}$.
The phase functions corresponding to $\eta_Y$ and $\eta'_Y$ are denoted by
$\alpha_Y$ and $\alpha'_Y$ respectively.
The proof of the following lemma is parallel
to that in \cite{Seidel_K3}:

\begin{lem}[cf. {\cite[Lemma 7.12]{Seidel_K3}}] \label{lem:7.12}
  For any $\epsilon > 0$, there exists $\delta > 0$ such that 
  if $\|y\| < \delta$ and $p(y) \ne 0$ then 
  \[
    \left| \frac 1{2\pi} \arg (\alpha_Y(\Lambda) / \alpha'_Y(\Lambda)) \right|
    < \epsilon
  \]
  for all $\Lambda \in \frakL_{Y,y}$.
\end{lem}

Let 
$
 H(y) = - \frac 12 |p(y)|^2
$
and consider its Hamiltonian vector field $X$ and flow $\phi_t$.
For a regular value $r$ of $\mu$, the induced function, 
Hamiltonian vector field, and its flow on $Y^{\red}$ are denoted by
\[
  H^{\red}(y^{\red}) = - 2^{k-3} \frac{r_2 \dots r_k}{r_1}
                      q(r_1, \dots, r_k, y_{k+1}, \dots, y_{n+1}),
\]
$X^{\red}$, and $\phi_t^{\red}$ respectively.
We write the complex structure on $Y^{\red}$
induced from $J'_Y$ as $J'_{Y^{\red}}$.
Then $\eta'_{Y^{\red}}$ gives a $J'_{Y^{\red}}$-complex volume form
on $Y^{\red}$.
Let $\alpha'_{Y^{\red}}$ be the phase function corresponding to 
$\eta'_{Y^{\red}}$.
The proof of the following lemma is
the same as in \cite{Seidel_K3}:


\begin{lem}[cf. {\cite[Lemma 7.13]{Seidel_K3}}] \label{lem:7.13}
  For any $\epsilon > 0$, there is $\delta > 0$ such that for
  $\|r \| < \delta$, $r_2\dots r_k/r_1 < \delta$, $\|y^{\red}\| < \delta$,
  and $|t| < \delta r_1 / r_2 \dots r_k$, 
  $\phi^{\red}_t$ is well-defined and
  \[
    | \tilde{\alpha}'_{\phi_t^{\red}} (\Lambda^{\red}) | < \epsilon
  \]
  for any Lagrangian subspace $\Lambda^{\red}$.
\end{lem}

Now we prove the following:

\begin{lem}[cf. {\cite[Lemma 7.14]{Seidel_K3}}] \label{lem:7.14}
  For any $\epsilon > 0$, there is $\delta_1 > \delta_2 > 0$ 
  such that if
  $\|y \| < \delta_1$, $|y_1| > \delta_2$, $0 < |p(y)| < \delta_1$
  and $|t|  < \delta_1 |p(y)|^{-2}$, 
  then $\phi_t$ is well-defined and satisfies
  \[
    \left| \tilde{\alpha}'_{\phi_t}(\Lambda) 
    - \frac{2t}{2\pi} 
      \left( 1 + \frac{|y_1|^2}{|y_2|^2} + \dots + \frac{|y_1|^2}{|y_k|^2}
      \right)^{-1} \right| < n+1 + \epsilon
  \]
  for any $\Lambda \in \frakL_{Y, y}$.
\end{lem}

\begin{proof}
The proof of well-definedness of $\phi_t$ is parallel to \cite{Seidel_K3}.
Note that the condition $|y_1| > \delta_2$ is preserved under the flow
since $\phi_t$ is $T^k$-equivariant.
Let $H' = - \frac 12 |y_2 \dots y_k/y_1|^2$ and
\[
  X' = -\ii \left( \frac 1{|y_1|^2} + \dots + \frac 1{|y_k|^2} \right)^{-1}
       \left( - \frac{y_1}{|y_1|^2},  \frac{y_2}{|y_2|^2}, 
              \dots,  \frac{y_k}{|y_k|^2}, 0, \dots, 0 \right)
\]
be its Hamiltonian vector field.
Then $H(y) = H'(y)r(y)$ for some smooth function 
$r(y) = 1 + O(\|y\|)$.
By direct computation, we have
\begin{align*}
  \| dH' \| &\le C \left| \frac{y_2 \dots y_k}{y_1} \right|^2
                 \left( \frac 1{|y_1|^2} + \dots + \frac 1{|y_k|^2} \right) \\
            &\le C \left| \frac{y_2 \dots y_k}{y_1} \right|^2
                 \frac{k \|y \|^{2(k-1)}}{|y_1 \dots y_k|^2} 
             = C \frac{k \| y \|^{2(k-1)}}{|y_1|^4},
\end{align*}
which is bounded if $\|y \| < \delta_1$ and $|y_1| > \delta_2$.
Then 
\[
  \| dH - d H' \| \le |r-1| \| dH' \| + |H'| \|dr \|
                  \le C (\| y \| + |H'|),
\]
and this implies that $\| dH -dH' \|$ is small 
if $|H|$ is also sufficiently small. 
Hence we obtain
\begin{equation}
  \|X - X' \| < \epsilon
  \label{eq:est_X}
\end{equation}
for small $\delta_1$.
Take a Lagrangian subspace $\Lambda^{\red}$ in $T_{y^{\red}}Y^{\red}$
and consider a Lagrangian subspace given by
\[
  \Lambda = \ii y_1 \bR \oplus \dots \oplus \ii y_k \bR \oplus \Lambda^{\red}
  \subset T_yY.
\]
Then we have 
\[
  \alpha'_Y(\Lambda) = (-1)^k \frac{y_1^2}{|y_1|^2} \cdot
                       \alpha'_{Y^{\red}}(\Lambda^{\red}),
\]
and hence
\begin{align*}
  \tilde{\alpha}'_{\phi_t}(\Lambda) 
    &= \frac 1{2\pi} \int_0^t X \arg (\alpha'_Y((D\phi_{\tau}(\Lambda)) d \tau\\
    &= \frac 1{2\pi} \int_0^t X' \arg \frac{y_1^2}{|y_1|^2} d \tau
     + \frac 1{2\pi} \int_0^t (X-X') \arg \frac{y_1^2}{|y_1|^2} d \tau\\
    &\qquad + \frac 1{2\pi} \int_0^t X^{\red} \arg (\alpha'_{Y^{\red}}
             ((D\phi^{\red}_{\tau}(\Lambda^{\red})) d \tau.
\end{align*}
The third term is small from Lemma \ref{lem:7.13}.
The second term is bounded by
\[
  \frac 1{2\pi}\int_0^{t} \| X-X' \| 
    \left\|D \arg \frac{y_1^2}{|y_1|^2} \right\| d \tau,
\]
which is also small from (\ref{eq:est_X}) and the fact that
\[
  \left\|D \arg \frac{y_1^2}{|y_1|^2} \right\| \le C \| X \|
  = C \| dH \|
\]
is uniformly bounded.
Since $|y_1|^2$ is preserved under the flow, the first term is
\begin{align*}
  &\frac 1{2\pi} \int_0^t X' \arg \frac{y_1^2}{|y_1|^2} d \tau \\
  &\quad = \frac 1{2\pi}
       \left( \frac 1{|y_1|^2} + \dots + \frac 1{|y_k|^2} \right)^{-1}
       \int_0^t \frac 1{|y_1|^2} 
         \ii y_1 \partial_{y_1}\arg \frac{y_1^2}{|y_1|^2} d \tau \\
  &\quad  = \frac 1{2\pi}
       \left( \frac 1{|y_1|^2} + \dots + \frac 1{|y_k|^2} \right)^{-1}
       \frac {2t}{|y_1|^2} \\
  &\quad  = \frac {2t}{2\pi}
       \left( 1 + \frac {|y_1|^2} {|y_1|^2} + \dots 
       + \frac {|y_1|^2} {|y_k|^2} \right)^{-1}.
\end{align*}
Then we obtain 
\[
    \left| \tilde{\alpha}'_{\phi_t}(\Lambda) 
    - \frac{2t}{2\pi} 
      \left( 1 + \frac{|y_1|^2}{|y_2|^2} + \dots + \frac{|y_1|^2}{|y_k|^2}
      \right)^{-1} \right| < \epsilon.
\]
For arbitrary Lagrangian subspace $\Lambda_1$, the desired bound for
$\tilde{\alpha}'_{\phi_t}(\Lambda_1)$ is obtained from this and
the fact that 
\[
  |\tilde{\alpha}'_{\phi_t}(\Lambda_1) - 
  \tilde{\alpha}'_{\phi_i}(\Lambda)| < n+1
\]
(see \cite[Lemma 6.11]{Seidel_K3}).
\end{proof}

Let $Z$ be the horizontal lift of $-\ii \zeta \partial_{\zeta}$,
and $\psi_t$ be its flow.
Then there is a positive function $f$ such that $Z = fX$,
and hence $\psi_t(y) = \phi_{g_t(y)}(y)$ for
\[
  g_t(y) = \int_0^t f(\psi_{\tau}(y)) d \tau.
\]
By the same argument as in \cite{Seidel_K3}, 
we have:

\begin{lem}[cf. {\cite[Lemma 7.15]{Seidel_K3}}]
  For any $d > 0$ and $\epsilon > 0$, there is $\delta > 0$ 
  such that for $\zeta \in \bC$ with $0<|\zeta| < \delta$ and
  $y \in Y_{\zeta}= p^{-1}(\zeta)$ with $\|y\| < \delta$,
  the $d$-fold monodromy $h_{\zeta}^d$ is well-defined, 
    $\epsilon / |\zeta|^2 > 2 \pi d$, and satisfies
  \[
    g_{2 \pi d}(y) \le \epsilon / |\zeta|^2.
  \]
\end{lem}

\begin{proof}[Proof of Proposition \ref{Seidel7.16}]
Let $\eta_{Y_{\zeta}} = \eta_Y / (d \zeta/ \zeta^2)$ be a 
complex volume form on $Y_{\zeta}$,
and $\alpha_{Y_{\zeta}}$ be the corresponding phase function.
Take $\Lambda \in \frakL_{Y,y}$ such that $Dp (\Lambda) = a \bR$
for $a \in U(1)$, and set 
$\Lambda^v = \Lambda \cap \ker Dp \in \frakL_{Y_{\zeta}, y}$.
Then 
\begin{equation}
  \alpha_{Y_{\zeta}}(\Lambda^v) = \frac{\zeta^4}{a^2 |\zeta|^4} 
  \alpha_Y(\Lambda).
  \label{eq:phase}
\end{equation}
We consider a Lagrangian subspace $\Lambda^v \in \frakL_{Y_{\zeta}, y}$
such that $Dp (\Lambda^v) = \ii \zeta \bR$, and containing
the tangent space of the torus action on $Y_{\zeta}$.
Then $\Lambda^v$ has the form
\[
  \Lambda^v = (\ii y_1 \bR \oplus \dots \oplus \ii y_k \bR
              \oplus \Lambda^{\red}) \cap \ker Dp.
\]
Let $\Lambda = \Lambda^v \oplus Z_y \bR \in \frakL_{Y,y}$.
Since $Z$ is the horizontal lift of 
$-\ii \zeta \partial_{\zeta} \in T_{\zeta}(\ii \zeta \bR)$, 
$Z_{\psi_t(y)}$ is contained in $D\psi_t(\Lambda)$, and hence we have
\[
  D(\psi_t|_{Y_{\zeta}}) (\Lambda^v)
  = D\psi_t(\Lambda) \cap \ker(Dp).
\]
From this and (\ref{eq:phase}) we have
\[
  \alpha_{\psi_t|_{Y_{\zeta}}}(\Lambda^v)
  = e^{-2t} \alpha_{\psi_t}(\Lambda).
\]
Combining this with Lemma \ref{lem:7.12} and \ref{lem:7.14}, we obtain
\begin{align*}
  \tilde{\alpha}_{h^d_{\zeta}} (\Lambda^v)
  &= \tilde{\alpha}_{g_{2\pi d}(y)} (\Lambda) - 2d \\
  &\le \tilde{\alpha}'_{g_{2\pi d}(y)} (\Lambda) - 2d + \epsilon \\
  &\le 2d \left( \left( 1 + \frac{|y_1|^2}{|y_2|^2} + \dots + \frac{|y_1|^2}{|y_k|^2}
      \right)^{-1} -1 \right)  + \epsilon \\
  &= - 2d \frac{\frac{1}{|y_2|^2} + \dots + \frac{1}{|y_k|^2}}
     {\frac{1}{|y_1|^2} + \frac{1}{|y_2|^2} + \dots + \frac{1}{|y_k|^2}}
       + \epsilon \\
  &\le -2d \frac{1}{1+ |\zeta|^2/|y_3|^{2(k-1)}}  + \epsilon
\end{align*}
if $|y_2| \le |y_3| \le \dots \le |y_k|$.
\end{proof}


Now we discuss gluing of the local models.
Let $X=\bP^{n+1}_{\bC}$
equipped with the standard complex structure $J_X$, 
the K\"{a}hler form $\omega_X$
and the anticanonical bundle $o_X = \scK^{-1}_X = \scO (n+2)$
as in Section \ref{sc:fuk}.
For $\sigma_{X,\infty} = x_1 \cdots x_{n+2}$ and a generic section
$\sigma_{X,0} \in H^0(\bP^{n+1}_{\bC}, \scO(n+2))$,
we consider a pencil of Calabi-Yau hypersurfaces defined by
\[
  X_z = \{ \sigma_{X,0} - z \sigma_{X,\infty} = 0 \} = p_X^{-1}( 1/z ),
\]
where $p_X = \sigma_{X,\infty}/\sigma_{X,0}$.
Let $C_i = \{ x_i = 0 \} \cong \bP^{n}_{\bC}$,
$i=1, \dots , n+2$ be the irreducible components of $X_{\infty}$
and set $C_0 = X_0$.
We assume that $\sigma_{X,0}$ is generic so that
the divisor $X_0 \cup X_{\infty}$
is normal crossing.
For $I \subset \{0, 1, \dots, n+2 \}$, we write
$C_I = \bigcap_{i \in I} C_i$ and
$C_I^{\circ} = C_I \setminus \bigcup_{J \supsetneq I} C_J.$
We will deform $\omega_X$ in such a way that it satisfies 
Assumption \ref{ass:local_model1} 
(resp. Assumption \ref{ass:local_model})
near $C_I$ with $0 \notin I$ (resp. $0 \in I$).

\begin{prop}
For each $I$, there exists a tubular neighborhood $U_I$ 
of $C_I$ in $\bP^{n+1}_{\bC}$ and
a fibration structure $\pi_I : U_I \to C_I$ such that
for each $p \in C_I$ 
the tangent space $T_p \pi_I^{-1}(p)$ of the fiber
is a complex subspaces in $T_pX$.
Moreover $\pi_I$ and $\pi_J$ are compatible
if $I \subset J$.
\end{prop}

See \cite[Proposition 7.1]{Ruan_II} 
for the definition of the compatibility.
This proposition is a weaker version of \cite[Proposition 7.1]{Ruan_II} 
in the sense that each fiber $\pi_I^{-1}(p)$ is required to be 
holomorphic only at $p \in C_I$.

\begin{proof}
For each $I$ we take a tubular neighborhood $U_I$ of $C_I$, and
consider an open covering $\{V_{\alpha} \}_{\alpha \in A}$ of
$\bigcup_I U_I$ satisfying
\begin{itemize}
  \item for each $\alpha \in A$, there exists a unique subset
        $I_{\alpha}$ in $\{0, 1, \dots, n+1\}$ such that 
        $V_{\alpha} \cap C_{I_{\alpha}} \ne \emptyset$ and
        $V_{\alpha} \cap C_J = \emptyset$ for all $J$ with
        $|J| > |I_{\alpha}|$, 
    \item $V_{\alpha}$ is a tubular neighborhood of 
        $V_{\alpha} \cap C_{I_{\alpha}}$, and
  \item for each $\alpha$, there exits a unique 
        $J_{\alpha} \supset I_{\alpha}$ such that if $V_{\alpha'}$
        intersects $V_{\alpha}$ and $|I_{\alpha'}| > |I_{\alpha}|$ 
        then $I_{\alpha} \subset I_{\alpha'} \subset J_{\alpha}$.
\end{itemize}
We take holomorphic coordinates 
$(w_{\alpha}, z_{\alpha}) = 
 (w_{\alpha}^1 \dots, w_{\alpha}^{n+1-|I_{\alpha}|},
 z_{\alpha}^1, \dots, z_{\alpha}^{|I_{\alpha}|})$ on $V_{\alpha}$ 
such that $C_{I_{\alpha}}$ is given by $z_{\alpha} = 0$ and $w_{\alpha}$ 
gives a coordinate on $C_{I_{\alpha}} \cap V_{\alpha}$,
and satisfying the following property: 
the projection $\pi_{\alpha} : V_{\alpha} \to C_{I_{\alpha}}$,
$(w_{\alpha}, z_{\alpha}) \mapsto w_{\alpha}$ is compatible
with $\pi_J$ for each $J \supset I_{\alpha}$.
Let $\{\rho_{\alpha} \}_{\alpha \in A}$ be a partition of unity
associated to $\{V_{\alpha} \}$.

Fix $p \in C_I^{\circ}$, and set 
$A_p := \{ \alpha \in A \, | \, p \in V_{\alpha} \}$.
Note that $I_{\alpha} \supset I$ for any $\alpha \in A_p$.
Take $\alpha_0 \in A$ such that 
$V_{\alpha_0} \cap V_{\alpha} \ne \emptyset$ for $\alpha \in A_p$
and $I_{\alpha_0} = J_{\alpha}$ is maximal.
Rename the coordinates on $V_{\alpha}$, $\alpha \in A_p$ so that 
the projection $\pi'_{\alpha} : V_{\alpha} \to C_I$ is given by
$(w'_{\alpha}, z'_{\alpha}) \mapsto w'_{\alpha}$.
Let
\[
  \operatorname{pr} : TV_{\alpha_0}|_{C_I} = 
  \vspan_{\bC} \left\{ \frac{\partial}{\partial w'^i_{\alpha_0}} \right\}
  \oplus
  \vspan_{\bC} \left\{ \frac{\partial}{\partial z'^j_{\alpha_0}} \right\}
  \longrightarrow
  \Ker d \pi'_{\alpha_0} =
  \vspan_{\bC} \left\{ \frac{\partial}{\partial z'^j_{\alpha_0}} \right\}
\]
be the projection.
After a coordinate change which is linear in $z'_{\alpha}$, we assume that 
$\operatorname{pr} (\partial / \partial z'^j_{\alpha}) = 
\partial / \partial z'^j_{\alpha_0}$ for each $j$.
Define 
\[
  E_{I, p} =  \vspan_{\bC} \Biggl\{
  \sum_{\alpha} \rho_{\alpha}(p) \frac{\partial}{\partial z'^j_{\alpha}} 
  \Biggm| j = 1, \dots, |I|
  \Biggr\}.
\]
Then $E_I = \bigcup_{p \in C_I} E_{I,p} \subset TX|_{C_I}$
is a complex subbundle which gives a splitting of 
$TX|_{C_I} \to \scN_{C_I/X} = TX|_{C_I}/TC_I$.
After shrinking $U_I$ if necessary, we obtain a fibration
$\pi_I : U_I \to C_I$ such that 
$T_p \pi_I^{-1}(p) = E_{I,p}$.
\end{proof}
Set
$
 U_I^{\circ} = \pi_I^{-1} (C_I^{\circ}).
$
We prove a weaker version of \cite[Theorem 7.1]{Ruan_II}.

\begin{prop} \label{prop:Ruan7.1}
There exists a K\"ahler form $\omega_X'$ in the class $[\omega_X]$ 
such that 
\begin{enumerate}
  \renewcommand{\labelenumi}{(\roman{enumi})}
  \item it tames $J_X$, and compatible with $J_X$ on $\bigcup_I C_I$,
  \item $\omega'_X = \omega_X$ outside a neighborhood of 
        $\mathrm{Sing} (X_0 \cup X_{\infty}) = \bigcup_{|I| \ge 2} C_I$,
  \item $C_i$'s intersect orthogonally, and
  \item each fiber of $\pi_I : U_I \to C_I$ is orthogonal to $C_I$.
\end{enumerate}
\end{prop}

\begin{proof}
It is shown in \cite[Lemma 1.7]{Seidel_LES} and 
\cite[Lemma 4.3]{Ruan_II} that $\omega_X$ can be modified locally 
so that it is standard near the lowest dimensional stratum 
$\bigcup_{|I|=n+1} C_I$.
We deform the symplectic form inductively to obtain $\omega'_X$

Fix $I \subset \{0, 1, \dots, n+1\}$ and 
take a distance function $r : X \to \bR_{\ge 0}$ from $C_I$,
i.e., $C_I = r^{-1}(0)$.
Fix a local trivialization of $o_X|_{U_I}$ by a section which
has unit pointwise norm and parallel in the radial direction of the fibers
of $\pi_I$, and let $\theta_X$ denote the connection 1-form. 
Then we have $\theta_X - \pi_I^*(\theta_X|_{TC_I}) = O(r)$.

Let $\pi : NC_I \to C_I$ be the symplectic normal bundle, 
i.e., $N_p C_I \subset T_pX$ is the orthogonal complement 
of $T_pC_I$ with respect to the symplectic form.
Let $\omega_N$ be the induced symplectic form on the fibers of
$NC_I$.
From the symplectic neighborhood theorem, a neighborhood of
$C_I$ is symplectomorphic to a neighborhood of the zero section
of $NC_I$ equipped with the symplectic form
$\pi^* (\omega_X|_{C_I}) + \omega_N$.
Identifying $NC_I$ with $E_I$, we obtain a symplectic form
$\omega_{U_I}$ on $U_I$ satisfying (i) and (iv).
Note that $\omega_{U_I}$ and $\omega_X$ coincide only on 
$TC_I$ in general.
Let $\theta_{U_I}$ be a connection 1-form on $o_X|_{U_I}$
such that $d \theta_{U_I} = \omega_{U_I}$ and
$\theta_{U_I}|_{TC_I} = \theta_X|_{TC_I}$.
We define $\eta = \theta_X - \theta_{U_I}$.
Then $\eta = 0$ on $C_I$.
Fix a constant $\delta >0$ such that $\{r \le \delta \} \subset U_I$
and take $C>0$ satisfying
\[
  \begin{cases}
    C^{-1} \omega_X \le t \omega_{U_I} + (1-t) \omega_X
    \le C \omega_X, \quad t \in [0,1],\\
    \| \eta \| \le Cr, \\
    \| dr \| \le C
  \end{cases}
\]
on $\{r \le \delta \}$.
Let $h : \bR \to \bR_{\ge 0}$ be a smooth function satisfying
\begin{itemize}
  \item $\lim_{s \to -\infty} h(s) = 1$,
  \item $h(s)=0$ for $s \ge \log \delta$, and
  \item $- 1/(2C^3) \le h'(s) \le 0$,
\end{itemize}
and set $f = h( \log r)$.
We define 
\[
  \theta' = \theta_X - f \eta
          = f \theta_{U_I} + (1-f) \theta_X
\]
and
\begin{align*}
  \omega' 
  &:= d \theta' 
  = f \omega_{U_I} + (1-f) \omega_X - df \wedge \eta \\
  &= f \omega_{U_I} + (1-f) \omega_X - h' dr \wedge \frac{\eta}{r}.
\end{align*}
Then $\omega'$ is compatible with $J_X$ along $C_I$ and the fibers 
of $\pi_I$ intersect $C_I$ orthogonally.
From the choice of $h$, we have
\[
  \| df \wedge \eta \| \le 
  \frac 1{2C^3} \cdot C \cdot C = \frac 1{2C},
\] 
which implies that $\omega'$ tames $J_X$, and hence it is non-degenerate.

By applying the argument in \cite[Lemma 1.7]{Seidel_LES} or
\cite[Lemma 4.3]{Ruan_II} to each fiber of $\pi_I$,
we can modify $\omega'$ to make $\omega'|_{\pi_I^{-1}(p)}$ 
standard at each $p \in C_I$, which means that 
$C_J$'s intersect orthogonally along $C_I$.
\end{proof}

Next we construct local torus actions.
Set $\scL_i = \scO (1) = \scO(C_i)$ for $i = 1, \dots, n+2$ and
$\scL_0 = \scO (n+2) = \scO(C_0)$.
Note that the normal bundle of $C_I$ is given by
\[
  \scN_{C_I/X} = \bigoplus_{i \in I} \scL_i |_{C_I}.
\]

For each $I = \{ i_1 < \dots < i_k \} \subset \{0, 1, \dots, n+2 \}$,
we define a $T^k$-action on $U_I^{\circ}$ as follows.
First we consider the case $0 \not\in I$.
Then we may assume that 
$(\prod_{j \not\in I \cup \{0\}} x_j)/ \sigma_{X,0} \ne 0$ on $U_I^{\circ}$
(after making $U_I$ smaller if necessary).
Then 
\[
  \bullet \otimes \frac{\prod_{j \not\in I \cup \{0\}} x_j}{\sigma_{X,0}} :
  \scL_{i_k} |_{U_I^{\circ}} \longrightarrow
  \scL_{i_k} \otimes \scL_0^{-1} \otimes 
  \Biggl. \bigotimes_{j \not\in I \cup \{0\}} \scL_j 
  \Biggr|_{U_I^{\circ}} \cong 
  \scO (1-k) |_{U_I^{\circ}}
\]
is an isomorphism, and thus we have
\[
  \scN_{C_I/X} |_{C_I^{\circ}}
  \cong \scN_I |_{C_I^{\circ}},
\]
where 
\begin{align*}
  \scN_I &:= \scL_{i_1} \oplus \dots \oplus \scL_{i_{k-1}} \oplus
             \biggl(  \scL_{i_k} \otimes \scL_0^{-1} \otimes 
             \bigotimes_{j \not\in I \cup \{0\}} \scL_j \biggr) \\
         &\cong \underbrace{\scO (1) \oplus \dots \oplus \scO (1)}_{k-1} 
                \oplus \scO (1-k).
\end{align*}
We identify $U_I^{\circ}$ with 
a neighborhood of the zero section of $\scN_I |_{C_I^{\circ}}$
by a map $\nu_I : U_I^{\circ} \to \scN_I |_{C_I^{\circ}}$
obtained by combining
\[
  \left( x_{i_1}, \dots, x_{i_{k-1}}, 
  \frac{x_{i_k} \prod_{j \not\in I \cup \{0\}} x_j}{\sigma_{X,0}} \right)
  : U_I^{\circ} \longrightarrow \scN_I
\]
with parallel transport along the fibers of 
$\pi_I : U_I^{\circ} \to C_I^{\circ}$.
The torus action on $U_I^{\circ}$ is defined to be the 
pull back the natural $T^k$-action on $\scN_I |_{C_I^{\circ}}$.
By construction, 
\begin{equation} \label{eq:comm_diag}
\begin{psmatrix}[rowsep=1cm]
 U_i^\circ & & \scN_I|_{C_I^\circ} \\
  & \bC &
\end{psmatrix}
\psset{shortput=nab,arrows=->,nodesep=3pt,labelsep=3pt}
\small
\ncline{1,1}{1,3}^{\nu_I}
\ncline{1,1}{2,2}_{p_X}
\ncline{1,3}{2,2}
\end{equation}
is commutative, where the right arrow is the natural map
\[
  \scN_I = \scO (1) \oplus \dots \oplus \scO (1)
  \oplus \scO (1-k) \longrightarrow \bC,
  \quad 
  (\zeta_1, \dots, \zeta_k) \longmapsto 
    \zeta_1 \dots \zeta_k.
\]
Hence $p_X = \sigma_{X, \infty}/\sigma_{X,0}$ is $T^k$-equivalent on $U_I^{\circ}$:
\[
  p_X (\rho_{I,s}(x)) = e^{\ii (s_1 + \dots + s_k)} p_X (x).
\]
Next we consider the case where $i_1 = 0 \in I$.
In this case we set 
\begin{align*}
  \scN_I &:= \scL_{i_1} \oplus \dots \oplus \scL_{i_{k-1}} \oplus
             \biggl(  \scL_{i_k} \otimes 
             \bigotimes_{j \not\in I} \scL_j \biggr) \\
         &\cong \scO (n+2) \oplus 
                \underbrace{\scO (1) \oplus \dots \oplus \scO (1)}_{k-2} 
                \oplus \scO (n+4-k).
\end{align*}
Assuming $\prod_{j \not\in I} x_j \ne 0$ on $U_I^{\circ}$, 
we have an isomorphism
\[
  \bigoplus_{i \in I} \scL_i |_{U_I^{\circ}} \longrightarrow 
  \scN_I |_{U_I^{\circ}}.
\]
By using
\[
  \left( \sigma_{X,0}, x_{i_2}, \dots, x_{i_{k-1}}, 
  x_{i_k} \prod_{j \not\in I} x_j \right) :
  U_I^{\circ} \longrightarrow \scN_I,
\]
we have a map $\nu_I : U_I^{\circ} \to \scN_I |_{C_I^{\circ}}$ 
identifying $U_I^{\circ}$ with a neighborhood the zero section, 
which gives a $T^k$-action on $U_I^{\circ}$ as above.
We also have a similar commutative diagram (\ref{eq:comm_diag})
where the right arrow in this case is
\[
   \scO (n+2) \oplus \scO (1) \oplus \dots \oplus \scO (1)
    \oplus \scO (n+4-k)
    \longrightarrow  \bC, \quad
    (\zeta_1, \dots, \zeta_k) \longmapsto 
    \frac{\zeta_2 \dots \zeta_k}{\zeta_1}.
\]
This means that $p_X$ is $T^k$-equivariant on $U_I^{\circ}$:
\[
  p_X (\rho_{I,s}(x)) = e^{\ii (-s_1 + s_2 + \dots + s_k)} p_X (x).
\]
We can easily check the compatibility of the above torus actions.
For example, we consider the case where 
$I = \{0, 1, \dots, k-1 \} \supset J = \{1, \dots, l\}$.
Take coordinates $(w_1, \dots, w_{n+1})$ around a point in $C_I$ such that 
$(w_1, \dots, w_k)$ gives fiber coordinates of $\pi_I$ corresponding to 
\[
  (\sigma_{X,0}, x_1, \dots, x_{k-2}, x_{k-1} \cdots x_{n+2}) : U_I \to \scN_I.
\]
Then the torus action is given by
\[
  (w_1, \dots, w_n) \longmapsto
  (e^{\ii s_1} w_1, \dots, e^{\ii s_k} w_k, w_{k+1}, \dots, w_{n+1}).
\]
On the other hand, since $\nu_J : U_J^{\circ} \to \scN_J|_{C_J^{\circ}}$
is obtained from 
\[
  \left( 
  x_1, \dots, x_{l-1}, \frac{x_l \dots x_{n+2}}{\sigma_{X,0}} 
  \right)
  : U_J^{\circ} \longrightarrow \scN_J, 
\]
$\nu_J$ restricted to $U_I^{\circ} \cap U_J^{\circ} \subset U_J^{\circ}$ is
given by
\[
  \nu_J (w_1, \dots, w_{n+1}) 
   = \left(w_2, \dots , w_l, \frac{w_{l+1} \dots w_k}{w_1} \right).
\]
This means that 
the torus action induced from $\rho_J$ is given by
\[
  (w_1, \dots, w_{n+1}) \longmapsto
  (w_1, e^{\ii s_2} w_2, \dots, e^{\ii s_{l+1}} w_{l+1}, w_{l+2}, \dots, w_{n+1}).
\]
(Note that $(w_1, w_{l+2}, \dots, w_{n+1})$ is a coordinate on 
the base $C_J \cap U_I$.) 
Other cases can be checked in similar ways.

By using the same argument as in \cite[Lemma 7.20]{Seidel_K3}, 
we have
\begin{prop}
There exists a K\"ahler form $\omega_X''$ in the class $[\omega_X]$ 
satisfying the conditions in Proposition \ref{prop:Ruan7.1}, and 
$\omega_X''|_{U_I^{\circ}}$ is invariant under the torus action
$\rho_I$ for each $I$.
\end{prop}

We fix $x \in C_I^{\circ}$ with $|I| = k$ and take a neighborhood
$U_x \subset U_I^{\circ}$ of $x$.
Let $Y \subset \bC^{n+1}$ be a small ball around the origin with the standard 
symplectic structure $\omega_Y$ and the $T^k$-action (\ref{eq:torus_action}).
Take a $T^k$-equivariant Darboux coordinate 
$\varphi : (U_x, \omega''_X) \to (Y, \omega_Y)$,
and define $J_Y = (\varphi^{-1})^* J_X$,
$p = (\varphi^{-1})^* p_X$, $\eta_Y = C (\varphi^{-1})^* \sigma_{X, \infty}^{-1}$,
where $C$ is a constant.
Then $(Y, \omega_Y, J_Y, \eta_Y, p)$ satisfies Assumption \ref{ass:local_model1} 
if $0 \not\in I$, or Assumption \ref{ass:local_model} if $0 \in I$
for a suitable choice of $C$.
%
%
Now we can follow the argument of \cite[Proposition 7.22]{Seidel_K3}
to complete the proof of Proposition \ref{prop:negativity}.

\section{Sheridan's Lagrangian as a vanishing cycle}
 \label{sc:sheridan_vc}

An $n$-dimensional pair of pants is
defined by
$$
 \scP^n
  = \lc [z_1 : \cdots : z_{n+2}] \in \bP_\bC^{n+1}
      \mid z_1 + \cdots + z_{n+2} = 0, \ 
           z_i \ne 0, \ i = 1, \ldots, n+2 \rc,
$$
equipped with the restriction
of the Fubini-Study K\"{a}hler form on $\bP_\bC^{n+1}$.
It is the intersection
of the hyperplane
$
 H = \{ z_1 + \cdots + z_{n+2} = 0 \}
$
with the big torus
$
 T
$
of $\bP_\bC^{n+1}$.
Sheridan \cite{Sheridan_pants} perturbs the standard double cover
$
 S^n \to H_\bR
$
of the real projective space $H_\bR \cong \bP_\bR^n$
by the $n$-sphere
slightly to obtain an exact Lagrangian immersion
$
 i : S^n \to \scP^n.
$
The real part $\scP^n \cap H_\bR$ of the pair of pants
consists of $2^{n+1} - 1$ connected components $U_K$
parametrized by proper subsets
$
 K \subset \{ 1, 2, \ldots, n + 2 \}
$
as
$$
 U_K = \{ [z_1 : \cdots : z_{n+2}] \in \scP^n \cap H_\bR \mid
  z_i / z_j < 0 \text{ if and only if } i \in K \text{ and } j \in K^c \}.
$$
Note that the set $\{ 1, \ldots, n + 2 \}$
has $2^{n+2} - 2$ proper subsets,
and one has $U_K = U_{K^c}$.
The inverse images of the connected component $U_K$
by the double cover
$
 S^n \to H_\bR
$
are the cells $W_{K, K^c, \emptyset}$ and
$W_{K^c, K, \emptyset}$
of the dual cellular decomposition
in \cite[Definition 2.6]{Sheridan_pants}.

The map
$
 p_\Mbar : \Mbar \to T
$
sending
$
 ( u_1, \ldots, u_{n+1},
    u_{n+2} = 1 / u_1 \cdots u_{n+1} )
$
to
$
 [z_1 : \cdots : z_{n+1} : 1]
$
for
$
 z_i = u_i \cdot u_1 \cdots u_{n+1},
$
$
 i = 1, \ldots, n + 1
$
is a principal $\Gamma_{n+2}^*$-bundle,
where the action of
$\zeta \cdot \id_{\bC^{n+2}} \in \Gamma_{n+2}^*$
sends
$
 (u_1, \ldots, u_{n+2})
$
to
$
 (\zeta u_1, \ldots, \zeta u_{n+2}).
$
The inverse map is given by
$
 u_1^{n+2} = z_1^{n+1}/z_2 \cdots z_{n+1}
$
and
$
 u_i = u_1 \cdot z_i / z_1
$
for $i = 2, \ldots, n+1$.
The restriction $p_{\Mbar_0} : \Mbar_0 \to \scP^n$
turns $\Mbar_0$ into a principal $\Gamma_{n+2}^*$-bundle
over the pair of pants.
One has
$$
 z_1 = - (1 + z_2 + \cdots + z_{n+1})
$$
on $\scP^n$,
so that
$
 u_1^{n+2}
  = (-1)^{n+1} f(z_2, \ldots, z_{n+1})
$
where
\begin{align} \label{eq:f}
 f(z_2, \ldots, z_{n+1})
  &= \frac{(1+z_2+\cdots+z_{n+1})^{n+1}}{z_2 \cdots z_{n+1}}.
\end{align}
The pull-back of Sheridan's Lagrangian immersion
by $p_{\Mbar_0}$
is the union of $n+2$ embedded Lagrangian spheres
$\{ L_i \}_{i=1}^{n+2}$ in $\Mbar_0$.

Recall that the {\em coamoeba}
of a subset of a torus $(\bCx)^{n+1}$
is its image by the argument map
$
 \Arg : (\bCx)^{n+1} \to \bR^{n+1} / 2 \pi \bZ^{n+1}.
$
Let $Z$ be the zonotope in $\bR^{n+1}$
defined as the Minkowski sum of
$
 \pi e_1, \ldots, \pi e_{n+1},
  - \pi e_1 - \cdots - \pi e_{n+1},
$
where $\{ e_i \}_{i=1}^{n+1}$ is
the standard basis of $\bR^{n+1}$.
The projection $\Zbar$ of $Z$ to $\bR^{n+1} / 2 \pi \bZ^{n+1}$ is
the closure of the complement
$(\bR^{n+1} / 2 \pi \bZ^{n+1}) \setminus \Arg(\scP^n)$
of the coamoeba of the pair of pants
\cite[Proposition 2.1]{Sheridan_pants},
and the argument projection of the immersed Lagrangian sphere
is close to the boundary of the zonotope
by construction
\cite[Section 2.2]{Sheridan_pants}.
The coamoeba of $\Mbar_0$ and the projections
of Lagrangian spheres $L_i$ are obtained
from those for $\scP^n$
as the pull-back by the $(n + 2)$-fold cover
\begin{equation} \label{eq:torus_covering}
\begin{array}{ccc}
 \bR^{n+1} / 2 \pi \bZ^{n+1} & \to & \bR^{n+1} / 2 \pi \bZ^{n+1} \\
 \vin & & \vin \\
 e_i & \mapsto & e_i + \sum_{j=1}^{n+1} e_j
\end{array}
\end{equation}
induced by $p_\Mbar : \Mbar \to T$.
It is elementary to see that
none of the pull-backs of the zonotope $\Zbar$
by the map \eqref{eq:torus_covering}
has self-intersections.
It follows that the argument projection of $L_i$
does not have self-intersections either,
which in turn implies that $L_i$ itself
does not have self-intersections,
so that $L_i$ is not only immersed but embedded.
We choose the numbering
on these embedded Lagrangian spheres
so that the argument projection of $L_i$ is close
to the boundary of the zonotope centered at
$
 [\frac{2 \pi}{n+2}(i, \ldots, i)]
   \in \bR^{n+1} / 2 \pi\bZ^{n+1}.
$

When $n = 1$,
the coamoeba of $\Mbar_0$ is the union of the interiors
and the vertices of six triangles
shown in Figure \ref{fg:coamoeba1}.
The projection of $L_{3}$ is also shown
as a solid loop in Figure \ref{fg:coamoeba1}.
The zonotope $\Zbar$ in this case is a hexagon,
whose pull-backs by the three-to-one map \eqref{eq:torus_covering}
are three hexagons
constituting the complement of the coamoeba.
Although the zonotope $\Zbar$ has self-intersections at its vertices,
none of its pull-backs has self-intersections
as seen in Figure \ref{fg:coamoeba1}.
The coamoeba of $\Mbar_0$ for $n=2$
is a four-fold cover
of the coamoeba of $\scP^2$ shown in
\cite[Figure 2(b)]{Sheridan_pants}.
\begin{figure}[htbp]
\begin{minipage}{.5 \linewidth}
\centering
\input{coamoeba1.pst}
\caption{The coamoeba}
\label{fg:coamoeba1}
\end{minipage}
\begin{minipage}{.5 \linewidth}
\centering
\input{coamoeba2.pst}
\caption{The cut and the thimble}
\label{fg:coamoeba2}
\end{minipage}
\end{figure}

Let
$
 \varpi : \Mbar_0 \to \bCx
$
be the projection
sending $(u_1, \cdots, u_{n+2})$ to $u_1$.

\begin{lemma}
The critical values of $\varpi$ are given by $(n+2)$ solutions
to the equation
\begin{equation} \label{eq:x1}
 u_1^{n+2} = (-1)^{n+1} (n+1)^{n+1}.
\end{equation}
\end{lemma}

\begin{proof}
The defining equation of $\Mbar_0$ in
$
 \Mbar = \Spec \bC[u_1^{\pm 1}, \ldots, u_{n+1}^{\pm 1}]
$
is given by
\begin{equation} \label{eq:Wfib}
 \sum_{i=1}^{n+1} u_i \cdot u_1 \cdots u_{n+1} + 1 = 0.
\end{equation}
By equating the partial derivatives
by $u_2, \ldots, u_{n+1}$ with zero,
one obtains the linear equations
$$
 u_i + \sum_{j=1}^{n+1} u_j = 0,
  \qquad i = 2, \ldots, n+1,
$$
whose solution is given by
$
 u_2 = \cdots = u_{n+1} = - u_1 / (n+1).
$
By substituting this into \eqref{eq:Wfib},
one obtains the desired equation \eqref{eq:x1}.
\end{proof}

Note that the connected component
\begin{align*}
 U_1
 &=U_{\{2, \ldots, n+2 \}}
 = \{ [z_1 : z_2 : \cdots : z_{n+1} : 1 ]
  \in \scP^n \mid
 (z_2, \ldots, z_{n+1}) \in (\bR^{>0})^n \}
\end{align*}
of the real part of the pair of pants
can naturally be identified with $(\bR^{>0})^n$.

\begin{lemma}
The function
$$
 f(z_2, \ldots, z_{n+1})
  = \frac{(1+z_2+\cdots+z_{n+1})^{n+1}}{z_2 \cdots z_{n+1}}
$$
has a unique non-degenerate critical point in $U_1 \cong (\bR^{>0})^n$
with the critical value $(n+1)^{n+1}$.
\end{lemma}
\begin{proof}
The partial derivatives are given by
\begin{align*}
 \frac{\partial f}{\partial z_2}
  &= ((n+1)z_2-(1+z_2+\cdots+z_{n+1}))
 \frac{(1+z_2+\cdots+z_{n+1})^n}{z_2^2 z_3 \cdots z_{n+1}}
\end{align*}
and similarly for $z_3, \ldots, z_{n+1}$.
By equating them with zero,
one obtains the equations
\begin{align*}
 (n+1)z_i-(1+z_2+\cdots+z_{n+1}) = 1,
   \qquad i = 2, \ldots, {n+1}
\end{align*}
whose solution is given by
$
 z_2 = \cdots = z_{n+1} = 1
$
with the critical value
$
  (n+1)^{n+1}.
$
\end{proof}

As an immediate corollary,
one has:

\begin{corollary} \label{cor:lefschetz1}
The inverse image of
$
 f : U_1 \to \bR
$
at $t \in \bR$
is
\begin{itemize}
\item empty if $t < (n+1)^{n+1}$,
\item one point if $t = (n+1)^{n+1}$, and
\item diffeomorphic to $S^{n-1}$ if $t > (n+1)^{n+1}$.
\end{itemize}
\end{corollary}

Recall that $f$ is introduced in \eqref{eq:f}
to study the inverse image
of the map $p : \Mbar_0 \to \scP^n$.

\begin{corollary} \label{cor:lefschetz2}
The inverse image $p^{-1}(U_1)$ consists of $n+2$
connected components $U_\zeta$
indexed by solutions to the equation
$
 \zeta^{n+2} = (-1)^{n+1}(n+1)^{n+1}
$
by the condition that $\zeta \in \varpi(U_\zeta)$.
\end{corollary}

One obtains an explicit description
of Lefschetz thimbles:

\begin{lemma} \label{lm:lefschetz3}
$U_\zeta$ is the Lefschetz thimble
for $\varpi : \Mbar_0 \to \bCx$
above the half line $\ell : [0, \infty) \to \bCx$
on the $x_1$-plane given by
$
 \ell(t) = t \zeta + \zeta.
$
\end{lemma}

\begin{proof}
The restriction of $\varpi$ to $U_\zeta$
has a unique critical point at
$
 (x_1, \ldots, x_{n+1})
  = \frac{\zeta}{n+1} ( n + 1, -1, \ldots, -1 ).
$
For $x = (x_1, \ldots, x_{n+1}) \in U_\zeta$
outside the critical point,
the fiber
$
 \scV_{x_1} =
  U_\zeta \cap \varpi^{-1}(x_1)
$
is diffeomorphic to $S^{n-1}$
by Corollary \ref{cor:lefschetz1},
and it suffices to show that
the orthogonal complement of
$T_x \scV_{x_1}$
in $T_x U_\zeta$ is orthogonal
to $T_x \varpi^{-1}(x_1)$
with respect to the K\"{a}hler metric $g$ of $\Mbar_0$.
Let $X \in T_x U_\zeta$ be a tangent vector
orthogonal to $T_x \scV_{x_1}$.
Then it is also orthogonal to $T_x \varpi^{-1}(x_1)$
since any element in $T_x \varpi^{-1}(x_1)$
can be written as $z Y$
for $z \in \bC$ and $Y \in T_x \scV_{x_1}$,
so that
$
 g(z Y, X) = z g(Y, X) = 0.
$
\end{proof}

The following simple lemma is a key
to the proof of Proposition \ref{prop:sheridan_vc}:

\begin{lemma} \label{lem:non-intrsection}
$U_\zeta$ for $\arg \zeta \ne \pm \frac{n+1}{n+2} \pi$
does not intersect $L_{n+2}$.
\end{lemma}

\begin{proof}
The map
$
  \bR^{n+1}/2 \pi \bZ^{n+1} \to \bR^{n+1}/ 2 \pi \bZ^{n+1}
$
induced from the map $p : \Mbar \to T$
is given on coordinate vectors by
$
 e_i \mapsto e_i + \sum_{j=1}^{n+1} e_j.
$
The inverse map is given by
$
 e_i \mapsto f_i
  = e_i - \frac{1}{n+2} \sum_{j=1}^{n+1} e_j,
$
so that
the argument projection of $L_{n+2}$ is close
to the boundary of the zonotope $Z_{n+2}$ generated by
$
 \pi f_1, \ldots, \pi f_{n+1},
  - \pi f_1 - \cdots - \pi f_{n+1}.
$
The argument projection of $U_\zeta$ consists
of just one point
$
 (\arg(\zeta), \arg(\zeta) + \pi , \ldots, \arg(\zeta) + \pi),
$
which is disjoint from $Z_{n+2}$
if $\arg \zeta \ne \pm \frac{n+1}{n+2} \pi$.
\end{proof}

The $n = 1$ case is shown in Figure \ref{fg:coamoeba2}.
Black dots
are images of $U_\zeta$ for
$
 \zeta = \sqrt[3]{4},
$
$
 \sqrt[3]{4} \exp(2 \pi \sqrt{-1}/3),
$
$
 \sqrt[3]{4} \exp(4 \pi \sqrt{-1}/3),
$
and white dots are images of $\Mbar_0 \setminus E$
defined below.
One can see that $L_3$ is contained in $E$ and
disjoint from $U_{\sqrt[3]{4}}$.

Now we use symplectic Picard-Lefschetz theory
developed by Seidel \cite{Seidel_PL}.
Put
$
 S = \bCx \setminus (-\infty, 0)
$
and let
$
 E = \varpi^{-1}(S)
$
be an open submanifold of $\Mbar_0$.
Note that both $V_{n+2}$ and $L_{n+2}$ are contained in $E$.
The restriction
$
 \varpi_E : E \to S
$
of $\varpi$ to $E$ is an exact symplectic Lefschetz fibration,
in the sense that all the critical points are non-degenerate
with distinct critical values.
Although $\varpi_E$ does not fit in the framework of Seidel
\cite[Section I\!I\!I]{Seidel_PL}
where the total space of a fibration
is assumed to be a compact manifold with corners,
one can apply the whole machinery of \cite{Seidel_PL}
by using the tameness of $\varpi_E$
(i.e., the gradient of $\| \varpi_E \|$ is bounded from below
outside of a compact set by a positive number)
as in \cite[Section 6]{Seidel_suspension}.
Let $\Fuk (\varpi_E)$ be the Fukaya category
of the Lefschetz fibration in the sense of Seidel
\cite[Definition 18.12]{Seidel_PL}.
It is the $\bZ / 2 \bZ$-invariant part
of the Fukaya category of the double cover
$\Etilde \to E$
branched along $\varpi_E^{-1}(*)$,
where $* \in S$ is a regular value of $\varpi_E$.
Different base points $* \in S$
lead to symplectomorphic double covers,
so that the quasi-equivalence class of $\Fuk (\varpi_E)$
is independent of this choice.
We choose $*$ to be a sufficiently large real number.
Let $(\gamma_1, \ldots, \gamma_{n+2})$
be a distinguished set of vanishing paths
chosen as in Figure \ref{fg:xplane_vp1}.
The pull-backs of the corresponding Lefschetz thimbles in $E$
by the double cover $\Etilde \to E$ will be denoted by
$(\Deltatilde_1, \ldots, \Deltatilde_{n+2})$,
which are called type (B) Lagrangian submanifolds
by Seidel \cite[Section (18a)]{Seidel_PL}.
On the other hand,
the pull-back of a closed Lagrangian submanifold of $E$,
which is disjoint from the branch locus,
is a Lagrangian submanifold of $\Etilde$
consisting of two copies of the original Lagrangian submanifold.
It also gives rise to an object of $\Fuk(\varpi_E)$,
which is called a type (U) Lagrangian submanifold by Seidel.
The letters
(B) and (U) stand for `branched' and `unbranched' respectively.

\begin{theorem}[{Seidel \cite[Propositions 18.13,
18.14, and 18.17]{Seidel_PL}}]
\ \vspace{-5mm}\\
\begin{itemize}
 \item
$(\Deltatilde_1, \ldots, \Deltatilde_{n+2})$
is an exceptional collection in $\Fuk (\varpi_E)$.
 \item
There is a cohomologically full and faithful $A_\infty$-functor
$\Fuk (E) \to \Fuk (\varpi_E)$.
 \item
The essential image of $\Fuk (E)$ is contained
in the full triangulated subcategory
generated by
$(\Deltatilde_1, \ldots, \Deltatilde_{n+2})$.
\end{itemize}
\end{theorem}

We abuse the notation
and use the same symbol $L_{n+2}$
for the corresponding object
in $\Fuk (\varpi_E)$.
The following lemma is a consequence
of Lemma \ref{lem:non-intrsection}:

\begin{lemma}
One has
$
 \Hom_{\Fuk (\varpi_E)}^*(\Deltatilde_i, L_{n+2}) = 0
$
for $i \ne 1, n + 2$.
\end{lemma}

\begin{proof}
For $2 \le i \le n + 1$,
move $* \in S$ continuously
from the positive real axis to
$$
 *' = \exp[(- n - 3 + 2 i) \pi \sqrt{-1} / (n+2)] \cdot *
$$
and move the distinguished set
$(\gamma_1, \ldots, \gamma_{n+2})$
of vanishing paths
in Figure \ref{fg:xplane_vp1}
to
$(\gamma_1', \ldots, \gamma_{n+2}')$
in Figure \ref{fg:xplane_vp2}
accordingly.
The corresponding double covers
$\Etilde$ and $\Etilde'$
are related by a Hamiltonian isotopy
sending type (B) Lagrangian submanifolds
$(\Deltatilde_1, \ldots, \Deltatilde_{n+2})$
of $\Etilde$
to type (B) Lagrangian submanifolds
$(\Deltatilde_1', \ldots, \Deltatilde_{n+2}')$
of $\Etilde'$.
It follows from Lemma \ref{lem:non-intrsection}
that the type (U) Lagrangian submanifold of $\Etilde'$
associated with $L_{n+2}$ does not intersect with $\Deltatilde_i'$.
This shows that
$
 \Hom_{\Fuk (\varpi_{E'})}^*(\Deltatilde_i', L_{n+2}) = 0,
$
which implies
$
 \Hom_{\Fuk (\varpi_E)}^*(\Deltatilde_i, L_{n+2}) = 0
$
by Hamiltonian isotopy invariance
of the Floer cohomology.
\begin{figure}
\scalebox{.9}{
\begin{minipage}{.35 \linewidth}
\centering
\scalebox{.9}{\input{xplane_vp1.pst}}
\caption{A distinguished set of vanishing paths}
\label{fg:xplane_vp1}
\end{minipage}}
\scalebox{.9}{
\begin{minipage}{.35 \linewidth}
\centering
\scalebox{.9}{\input{xplane_vp2.pst}}
\caption{Another distinguished set of vanishing paths}
\label{fg:xplane_vp2}
\end{minipage}}
\scalebox{.9}{
\begin{minipage}{.35 \linewidth}
\centering
\scalebox{.9}{\input{Pn_mp.pst}}
\caption{The matching path}
\label{fg:Pn_mp}
\end{minipage}}
\end{figure}
\end{proof}

It follows that $L_{n+2}$ belongs to the triangulated subcategory
generated by the exceptional collection
$(\Deltatilde_1, \Deltatilde_{n+2})$.
Since $L_{n+2}$ is exact,
the Floer cohomology of $L_{n+2}$ with itself is isomorphic
to the classical cohomology of $L_{n+2}$.

\begin{lemma}[{Seidel \cite[Lemma 7]{Seidel_GKQ}}]
Let $\scT$ be a triangulated category
with a full exceptional collection $(\scE, \scF)$
such that
$
 \Hom^*(\scE, \scF) \cong H^*(S^{n-1}; \bC),
$
and $L$ be an object of $\scT$
such that
$
 \Hom^*(L, L) \cong H^*(S^n; \bC).
$
Then $L$ is isomorphic
to the mapping cone
$
 \Cone(\scE \to \scF)
$
over a non-trivial element in $\Hom^0(\scE, \scF) \cong \bC$
up to shift.
\end{lemma}

This shows that $L_{n+2}$
is isomorphic to
$\Cone(\Deltatilde_1 \to \Deltatilde_{n+2})$
in $D^\pi \scF (\varpi_E)$
up to shift.
On the other hand,
it is shown in \cite[Section 5]{Futaki-Ueda_Pn}
that $V_{n+2}$ is isomorphic to the matching cycle
associated with the matching path $\mu_{n+2}$
shown in Figure \ref{fg:Pn_mp}
(cf. \cite[Figure 5.2]{Futaki-Ueda_Pn}).
Here, a {\em matching path} is a path
on the base of a Lefschetz fibration
between two critical values,
together with additional structures
which enables one to construct a Lagrangian sphere
(called the {\em matching cycle})
in the total space by arranging vanishing cycles
along the path
\cite[Section (16g)]{Seidel_PL}.
Since the matching path $\mu_{n+2}$ does not intersect
$\gamma_i$ for $i \ne 1, n+2$,
the vanishing cycle $V_{n+2}$ is also orthogonal
to $\Deltatilde_2, \ldots, \Deltatilde_{n+1}$
in $D^\pi \Fuk (\varpi_E)$.
It follows that $L_{n+2}$ equipped with a suitable grading
is isomorphic to $V_{n+2}$ in $\scF(E)$.
Note that any holomorphic disk in $\Mbar_0$
bounded by $L_{n+2} \cup V_{n+2}$
is contained in $E$,
since any such disk projects by $\varpi$
to a disk in $S$.
This shows that
the isomorphism $L_{n+2} \simto V_{n+2}$ in $\scF(E)$ extends
to an isomorphism in $\scF(\Mbar_0)$,
and the following proposition is proved:

\begin{proposition} \label{prop:sheridan_vc1}
$L_{n+2}$ and $V_{n+2}$ are isomorphic in $\scF(\Mbar_0)$.
\end{proposition}

Proposition \ref{prop:sheridan_vc} follows
from Proposition \ref{prop:sheridan_vc1}
by the $\Gamma_{n+2}^*$-action,
which is simply transitive
on both $\{ V_i \}_{i=1}^{n+2}$ and $\{ L_i \}_{i=1}^{n+2}$.

\begin{remark} \label{rm:compatibility}

Let $\dirFuk$ be the directed subcategory
of $\Fuk(M_0)$
consisting of the distinguished basis
$(\Vtilde_i)_{i=1}^{N}$
of vanishing cycles
of the exact Lefschetz fibration
$\pi_M : M \to \bC$;
$$
 \hom_{\dirFuk}(\Vtilde_i, \Vtilde_j) =
\begin{cases}
 \bC \cdot \id_{\Vtilde_i} & i = j, \\
 \hom_{\Fuk(M_0)}(\Vtilde_i, \Vtilde_j) & i < j, \\
 0 & \text{otherwise}.
\end{cases}
$$
It is also isomorphic
to the directed subcategory of $\Fuk(X_0)$,
since the compositions $\frakm_2$ are the same
on $\Fuk(M_0)$ and $\Fuk(X_0)$,
and higher $A_\infty$-operations $\frakm_k$
for $k \ge 3$ vanish
on the directed subcategories.
Symplectic Picard-Lefschetz theory
developed by Seidel
\cite[Theorem 18.24]{Seidel_PL}
gives an equivalence
$$
 D^b \dirFuk \cong D^b \Fuk(\pi_M)
$$
with the Fukaya category of the Lefschetz fibration $\pi_M$.
This provides a commutative diagram
$$
\begin{array}{ccc}
 \dirFuk & \hookrightarrow & \scF_q \\
 \rotatebox{90}{$\cong$} & & \rotatebox{90}{$\cong$} \\
 \dirC_{n+2} \rtimes \Gamma & \hookrightarrow & \psi^* \scS_q
\end{array}
$$
of $A_\infty$-categories,
where horizontal arrows are embeddings
of directed subcategories.
Combined with the equivalences
$D^b \dirFuk \cong D^b \Fuk(\pi_M)$,
$D^\pi (\scF_q \otimes_{\Lambda_\bN} \Lambda_\bQ)
 \cong D^\pi \Fuk(X_0)$,
$D^b (\dirC_{n+2} \rtimes \Gamma)
 \cong D^b \coh [\bP_\bC^n/\Gamma]$
and
$D^\pi (\scS_q \otimes_{\Lambda_\bN} \Lambda_\bQ))
 \cong D^b \coh Z_q^*$,
this gives the compatibility
of homological mirror symmetry
$$
 D^b \Fuk(\pi_M) \cong D^b \coh [\bP_\bC^n/\Gamma]
$$
for the ambient space
and homological mirror symmetry
$$
 D^\pi \Fuk(X_0) \cong \psihat^* D^b \coh Z_q^*
$$
for its Calabi-Yau hypersurface.
\end{remark}


\bibliographystyle{amsalpha}
\bibliography{bibs}

\noindent
Yuichi Nohara

Faculty of Education,
Kagawa University,
1-1 Saiwai-cho,
Takamatsu,
Kagawa,
760-8522,
Japan.

{\em e-mail address}\ : \  nohara@ed.kagawa-u.ac.jp
\ \vspace{0mm} \\

\noindent
Kazushi Ueda

Department of Mathematics,
Graduate School of Science,
Osaka University,
Machikaneyama 1-1,
Toyonaka,
Osaka,
560-0043,
Japan.

{\em e-mail address}\ : \  kazushi@math.sci.osaka-u.ac.jp
\ \vspace{0mm} \\

\end{document}